\documentclass[11pt,a4paper,oneside]{amsart}
\usepackage[utf8]{inputenc}
\usepackage[T1]{fontenc} 
\usepackage[pdftex]{graphicx}
\usepackage[english]{babel}
\usepackage{bold-extra,geometry,algorithm,amsmath,mathtools, microtype,url,cite}
\usepackage{listings}
\usepackage{amsmath}
\usepackage[foot]{amsaddr}
\usepackage{amssymb}
\usepackage{amsthm}
\usepackage{commath}
\usepackage[export]{adjustbox}
\usepackage{mathtools}
\usepackage{booktabs} 
\usepackage{siunitx}  
\usepackage{graphicx} 
\usepackage{sidecap}
\usepackage{caption}  
\usepackage{subfig}
\usepackage{bbm}

\newtheorem{definition}{Definition}[section]
\newtheorem{theorem}[definition]{Theorem}
\newtheorem{rem}[definition]{Remark}
\newtheorem{lemma}[definition]{Lemma}
\newtheorem{prop}[definition]{Proposition}

\usepackage{appendix}
\usepackage{enumitem}

\usepackage{hyperref}
\DeclareMathOperator{\Div}{Div}
\usepackage{soul} 
\usepackage{xcolor}
\definecolor{blu}{rgb}{0.1,0.0,0.9}
\definecolor{bordeaux}{rgb}{0.75,0.0,0.0}

\newcommand\blu{\color{black}}

\newcommand\R{\mathbb{R}}
\newcommand\N{\mathbb{N}}

\newcommand{\MM}{\mathcal{M}}
\newcommand{\PP}{\mathcal{P}}
\newcommand{\WW}{\mathcal{W}}
\newcommand{\supp}{\text{supp}}
\newcommand{\sgn}{\text{sgn}}
\newcommand{\Lip}{\text{Lip}}

\newcommand\eps{\epsilon}
\newcommand\weakstar{\overset{\ast}{\rightharpoonup}}
\usepackage{verbatim}

\title{A Nonlocal Aw-Rascle-Zhang system with linear pressure term}
\author{Debora Amadori}
\address{Dipartimento di Ingegneria e Scienze dell'Informazione e Matematica (DISIM), University of L'Aquila, L'Aquila, Italy} 
\email{debora.amadori@univaq.it}

\author{Felisia Angela Chiarello}
\address{Dipartimento di Ingegneria e Scienze dell'Informazione e Matematica (DISIM), University of L'Aquila, L'Aquila, Italy} 
\email{felisiaangela.chiarello@univaq.it}

\author{Gianmarco Cipollone}
\address{Dipartimento di Ingegneria e Scienze dell'Informazione e Matematica (DISIM), University of L'Aquila, L'Aquila, Italy} 
\email{gianmarco.cipollone@graduate.univaq.it {\rm (corresponding author)}}

\begin{document}
\begin{abstract}
In this paper, we study a nonlocal extension of the Aw-Rascle-Zhang traffic model, where the pressure-like term 
is modeled as a convolution between vehicle density and a kernel function. This formulation captures nonlocal driver
interactions and aligns structurally with the Euler-alignment system studied in \cite{Leslie-Tan}. 
Using a sticky particle approximation, we construct entropy weak solutions to the equation for the cumulative density and prove convergence of approximate solutions
to weak solutions of the nonlocal system. The analysis includes well-posedness, stability estimates, and an entropic selection principle.
\end{abstract}



\maketitle


\section{Introduction}
\setcounter{equation}{0}
Traffic flow models have been studied since many decades, following the growing scientific interest raised by the applications.
Within the class of macroscopic models, the classical LWR model \cite{LighthillWhitham, Richards} is based on a closure relation 
that expresses the velocity as a function of the density and has been widely used for applications in engineering because of its simplicity. 
However, it shows some limitations, as for instance it is not suitable to predict some real-life complex traffic phenomena such as phantom jams. 

Second order models have been introduced to overcome the limitations for first order models. One of the well-known second order traffic models 
is the so-called Aw-Rascle-Zhang model (ARZ), introduced in \cite{AR00,Zhang2002}, that reads as
\begin{equation}\label{eq:ARZ-local}
   \begin{cases}
        \partial_t\rho+\partial_x(\rho u)=0,\\
        \partial_t (u+ p(\rho)) +u\partial_x(u+p(\rho))=0\,.
    \end{cases}
\end{equation}
Here $\rho(x,t)\ge 0$ represents the macroscopic density of vehicles, i.e. the number of vehicles per unit length, 
$u(x,t)$ the mean velocity and $p(\rho)$ is a smooth increasing pressure-like function 
representing the attitude of drivers to adapt to the traffic conditions.

In this paper, motivated by some recent results about nonlocal traffic models, see for example\cite{ChiarelloGoatin, 2018Gottlich,
chiarello2020micro}, we introduce a nonlocal version of \eqref{eq:ARZ-local}  where the pressure term is replaced by a convolution
term between the density $\rho$ and a kernel function $\omega$. In particular, we focus our analysis on the case $p(\rho)=\rho$.
The system reads as  
\begin{equation}\label{ARZ first}
    \begin{cases}
        \partial_t\rho+\partial_x(\rho u)=0,\\
        \partial_t u+u\partial_xu=-(\partial_t+u\partial_x)(\omega\ast\rho)\,.
    \end{cases}
\end{equation}
The second equation describes the acceleration of drivers along their motion, and implies that the quantity 
\begin{equation}\label{def:psi}
 \psi:=u+\omega\ast\rho
\end{equation}
satisfies the equation 
\begin{equation}\label{eq:cons-of-psi}
    \partial_t\psi+u\partial_x\psi=0\,.
\end{equation}
Multiplying the first equation of the system by $\psi$ and the second one by $\rho$, the system can be written in conservation form as follows:
\begin{equation}\label{ARZ_conservative}
    \begin{cases}
        \partial_t\rho+\partial_x(\rho u)=0,\\
        \partial_t(\rho\psi)+\partial_x(\rho u\psi)=0.
    \end{cases}
\end{equation}
Throughout the paper we will assume that the initial total mass equals to 1 and then, by conservation of total mass, $\int_\R\rho(x,t)\,dx=1$ for all $t>0$. 
The convolution kernel $\omega$ in \eqref{ARZ first} is assumed to satisfy
\begin{equation*}
    \textbf{(H)} \qquad \omega:\R\to[0,+\infty)\,,\quad \omega \in W^{1,1}(\R)\cap W^{1,\infty}(\R)\,,\qquad \int_\R\omega\,dx=1\,. 
\end{equation*}
Moreover, we will denote by $\phi$ the derivative of $\omega$: 
\begin{equation}\label{phi-def}
    \phi:=\omega' \in L^1(\R)\cap L^\infty(\R)\,.
\end{equation}
It is interesting to compare \eqref{ARZ first} with the Euler-alignment system, 
coming from the theory of collective behavior, see \cite{Leslie-Tan},
\begin{equation}\label{eq:Euler_alignment}
\begin{cases}
    \partial_t \rho +\partial_x (\rho u)=0,\\
    \partial_t(\rho u)+ \partial_x (\rho u^2)=\rho (\phi \ast(\rho u))-\rho u (\phi \ast \rho)
\end{cases}
\end{equation}
where $\rho$ and $u$ represent the density and velocity, respectively, and the function $\phi$ is called 
\textit{communication protocol} governing the strength of the interactions between the agents. 

Indeed, we reformulate the second equation of \eqref{ARZ first} using the evaluations of the partial derivatives 
of the convolution products, i.e. 
\begin{align*}
    \partial_t(\omega\ast\rho)(x,t)
    &=\int_{-\infty}^\infty\partial_t\rho(y,t)\omega(x-y)\,dy=-\int_{-\infty}^\infty \partial_x(\rho u)(y,t)\omega(x-y)\,dy \\
    &
    =-\int_{-\infty}^\infty (\rho u)(y,t) \phi
    (x-y)\,dy =- (\phi \ast (\rho u)) (x,t),
\end{align*}
and 
\begin{align*}
   u \partial_x(\omega\ast\rho)(x,t)&= u(x,t)\int_{-\infty}^\infty \omega(x-y)\rho_y(y,t)\,dy \\
   &=u(x)\omega(x-y)\rho(y,t)|_{-\infty}^\infty+u(x,t)\int_{-\infty}^\infty\phi(x-y)\rho(y,t)\,dy \\ 
   &=u(\phi\ast\rho)(x,t).
\end{align*}
Now, the second equation of \eqref{ARZ first} reads
\begin{equation}\label{eq:momentum_eq}
    \partial_tu+u\partial_xu=\phi\ast(\rho u)-u(\phi\ast\rho);
\end{equation}
multiplying \eqref{eq:momentum_eq} by the density $\rho$ and summing the equation of the conservation of mass, we obtain
\begin{equation*}
    \partial_t(\rho u)+\partial_x(\rho u^2)=\rho(\phi\ast(\rho u))-\rho u(\phi\ast\rho).
\end{equation*}
Then, the nonlocal ARZ system \eqref{ARZ first} \textcolor{black}{with linear pressure} can be rewritten in the Euler-alignment form \eqref{eq:Euler_alignment}.

System \eqref{eq:Euler_alignment} with measure-valued density and bounded velocity has been studied in \cite{Leslie-Tan}, 
considering a communication protocol $\phi\in L^1_{loc}(\R)$ which is non-negative and even.
Appropriately chosen $\phi$ lead to a long-time behavior referred to as flocking, see \cite{Shvydkoy2021}. 
Moreover, the global regularity theory for smooth solutions has been developed in \cite{Carrillo2016} 
via a critical threshold condition for the regular case, i.e. a communication protocol which is Lipschitz continuous and bounded.
For weakly singular interaction, i.e. $\phi$ with an integrable singularity at the origin, details about smooth solutions can be found 
in \cite{Leslie2020}, while for strongly singular interaction, see \cite{Do2017, Kiselev2018, Shvydkoy2017}.
It is important to notice that system \eqref{eq:Euler_alignment} is the hydrodynamic version of the Cucker-Smale system of ODE's, 
see \cite{Cucker2007}.
The derivation of a kinetic formulation corresponding to the Cucker-Smale system can be found in \cite{Ha2008, Carrillo2010, Mucha2017}.

The purpose of this paper is to prove the well-posedness and stability estimates for system \eqref{ARZ_conservative}, and to provide an entropic selection principle. 
In particular, we develop a global well-posedness theory for weak solutions of \eqref{ARZ_conservative} with measure-valued density and bounded velocity. 
Our approach is based on an adaptation of the sticky particle approximation, originally introduced by Brenier and Grenier \cite{Brenier-Grenier} to analyze 
the 1D pressureless Euler equations and developed in \cite{Leslie-Tan}. 
The novelty with respect to the previous results in \cite{Leslie-Tan} is that we are considering an integrable and bounded kernel function $\phi$ which is not necessarily symmetric (even) nor positive. The shape of its antiderivative $\omega$, see \eqref{phi-def} and  \textbf{(H)}, can be chosen closer to that one of the anisotropic kernel functions usually used in literature for nonlocal traffic flow models such as in \cite{chiarello, chiarello2020micro, ChiarelloGoatin}.\\ 
The choice of a non-symmetric kernel function makes the analysis more involved and interesting; 
moreover it is coherent from the modeling point of view because drivers are supposed to adapt the velocity with respect to what happens 
mainly downstream in the traffic. It is worth quoting the following recent paper \cite{HT2024} in which the authors consider a nonlocal version 
of the ARZ model with nonlocality in the pressure density obtained as a convolution type term between a positive, bounded and non-increasing 
on $[0,\infty)$ kernel function. In particular, \cite{HT2024} is focused on the well-posedness of \textit{smooth solutions} and the phase-plane analysis.

{Other recent related papers are \cite{ACCGK2025}, \cite{CPSZ2024}, \cite{MarconiSpinolo2025} and \cite{WZ2025}.} In \cite{MarconiSpinolo2025} the authors establish existence, uniqueness 
and Ole\u{i}nik estimates, together with the singular limit, for a nonlocal version of the so-called Generalized Aw-Rascle-Zhang model (GARZ) \cite{FanHerty2014}, 
where the drivers' velocity depends on both the empty road velocity and the convolution of the density with an anisotropic kernel. 
In \cite{ACCGK2025} the authors analyze the well-posedness and the singular limit for another nonlocal version of the GARZ model 
in which the nonlocality is present in the flux function as a convolution-type term between the velocity function and an anisotropic kernel function.

{In \cite{CPSZ2024,WZ2025} the authors analyze measure-valued solutions to a non-local dissipative Aw-Rascle-Zhang model and an Euler-alignment system with matrix-valued communication $D^2K$ in a multidimensional setting, highlighting the connection between the two models. Compared to their assumptions, we assume $K$ to be monotone non-decreasing with $K(-\infty)=0$,  $K(+\infty)=1$ and with $K''=\phi \in L^1(\R)\cap L^\infty(\R)$; we do not assume $\phi$ to be even.
The present paper was developed independently of these works.}

The paper is organized as follows. 

- In Section~\ref{Sec:2} we introduce the microscopic approximation based on a nonlocal version of the sticky particle dynamics. 

- In Section~\ref{Sec:3} we introduce a scalar balance law linked to system \eqref{ARZ_conservative}, 
that involves the cumulative quantities $M$ and $Q$ introduced in \eqref{def-M-Q}, 
and we prove the well-posedness of the Cauchy problem for $M$, see \eqref{def:Cauchy_problem} and Theorem~\ref{theorem:entropy-solution}. 
The proof relies on the sticky particle approximation introduced in Section~\ref{Sec:2}.

- In Section~\ref{Sec:4} we use the result in Section~\ref{Sec:3} to prove the existence of a unique solution of system \eqref{ARZ_conservative} 
in the sense of measure and a stability estimate with respect to the initial data, see Theorem~\ref{th:weak_sol_ARZ}. 

Finally, two appendixes collect some analytical tools used in the paper.

\section{Sticky particle Cucker-Smale}\label{Sec:2}
\setcounter{equation}{0}
In this section, we recall the microscopic formulation of the well-known Cucker-Smale system of ODE's 
\cite{Cucker2007} i.e.
\begin{equation}\label{eq:C-S}
\begin{cases}
\dot x_i=v_i,\\
\dot v_i=\displaystyle{\sum_{j=1}^N} m_j\phi(x_i-x_j)(v_j-v_i),
\end{cases}
\end{equation}
assuming that the total mass is 1:
\begin{equation}\label{eq:mass_1}
\sum_{i=1}^N m_i=1.
\end{equation}
System \eqref{eq:C-S}-\eqref{eq:mass_1} allows for particle trajectories to cross. 
To avoid this, we consider its \textit{sticky particle} version, i.e.: if two or more particles collide, meaning that they have the same position 
for the first time, they remain stuck and start to travel together with the same velocity. 
This sticky particle model was originally proposed in \cite{Zeldovich70, Grenier1995ExistenceGP} 
to prove the existence of the so-called `atomic' weak solutions to the 1D pressurless Euler system.
System \eqref{eq:C-S} governs $N$ agents defined by three quantities: mass $m_i \geq 0$,  position $x_i$ 
and velocity $v_i$ for any $i=1,\dots,N$.
\newline
From the fact that the particles remain stuck after a collision, it is useful to define the following sets $J_i$:
\begin{equation*}
    J_i(t)=\{j\in 1,\dots, N: x_j(t)=x_i(t)\},\quad i=1,\dots,N\,,\quad t\ge 0.
\end{equation*}
From the definition of $J_i$, it follows that 
\begin{equation}\label{eq:inclusion}
    J_i(t)\subseteq J_i(s) \quad \mbox{ for } 
    0\le t\le s\,,\quad \forall\, i\in\{1,\dots,N\}.
\end{equation}
Having \eqref{eq:inclusion}, we can order the $N$ particles involved in the system because when two particles collide 
they can not cross each other. Thus, we write
\begin{equation}\label{eq:particle-do-not-cross}
    x_1(t)\leq x_2(t)\leq\dots\leq x_N(t),\quad\forall \,t\geq0.
\end{equation}

Moreover, it is convenient to set notation for the lowest and the highest indices in a given $J_i(t)$, thus
\begin{equation}\label{def:i-star}
    i_*(t)=\min J_i(t),\qquad i^*(t)=\max J_i(t).
\end{equation}
Then, the dynamics is described as follows:

\begin{itemize}
    \item when $t$ is not a collision time, the velocity and the acceleration of particles are governed by the following system  
\begin{equation}\label{eq:C-S omega}
    \begin{cases}
        \dot x_i=v_i,\\
        \dot v_i=\sum_{j=1}^Nm_j\phi(x_i-x_j)(v_j-v_i);
    \end{cases}
\end{equation}

 \item when $t$ is a collision time, the particles involved in a collision remains stuck and start to travel with the same velocity defined by
 \begin{equation}\label{def:common-velocity}
        v_i(t+):=\frac{\sum_{j\in J_i(t)}m_j v_j(t-)}{\sum_{j\in J_i(t)}m_j}.
    \end{equation}
\end{itemize}
We define the quantities $\psi_i$ that are a discrete analog of the quantity $\psi$ defined in \eqref{def:psi} thus we have: for those $t$ that are not collision times,
\begin{equation}\label{def:psi-i}
    \psi_i(t):=v_i(t)+\sum_{j=1}^Nm_j\omega(x_i(t)-x_j(t)).
\end{equation}
We give now  two properties of these quantities $\psi_i$.
\begin{prop}\label{prop:properties-psi_i}
    The following properties hold:
\begin{itemize}
    \item[(i)] If $t$ is not a collision time, then $\frac{d}{dt}\psi_i(t)=0$;
    \item[(ii)] If $t$ is a collision time, it holds
\begin{equation}\label{eq: conservation of rho-psi}
    \psi_i(t)=\frac{\sum_{j\in J_i(t)}m_j\psi_j(t-)}{\sum_{j\in J_i(t)}m_j}\,.
\end{equation}
\end{itemize}
\end{prop}
\begin{proof} 

\underline{Proof of $(i)$}. If $t$ is not a collision time, consider the definition \eqref{def:psi-i} of the $\psi_i$ 
and the definition \eqref{eq:C-S} of $\dot v_i$ by the Cucker-Smale dynamics, $i=1,...,N$.

By evaluating the time derivative and using that $\omega'=\phi$, we write 
\begin{align*}
    \frac{d}{dt}\psi_i(t)&=\frac{d}{dt}\left(v_i(t)+\sum_{j=1}^Nm_j\omega(x_i(t)-x_j(t))\right)\\ 
    &=\dot v_i(t)+\sum_{j=1}^Nm_j\phi(x_i(t)-x_j(t))(v_i(t)-v_j(t))\\
    &=\sum_{j=1}^Nm_j\phi(x_i(t)-x_j(t))(v_j(t)-v_i(t))+\sum_{j=1}^Nm_j\phi(x_i(t)-x_j(t))(v_i(t)-v_j(t))\\
    &=0.
\end{align*}

    \underline{Proof of $(ii)$}. We consider a particle $i\in\{1,\ldots,N\}$ at a collision  time $t$. 
    The aim is to prove that the following equality holds true
    \begin{equation}\label{eq:(ii)}
        \sum_{j\in J_i(t)}m_j\psi_j(t-)=\left(\sum_{j\in J_i(t)}m_j \right) \psi_i(t+).
    \end{equation}
Then, we rewrite the above equation \eqref{eq:(ii)} in the following way
\begin{equation*}
    \sum_{j=i_*(t)}^{i^*(t)}m_j\psi_j(t-)=\left(\sum_{j=i_*(t)}^{i^*(t)}m_j \right)\psi_i(t+).
\end{equation*}
    For any $k\in\{i_*(t),\dots,i^*(t)\}$, by the definition of $\psi_i$ in \eqref{def:psi-i} we have that
    \begin{equation}\label{eq:sum-barycentric}
        m_k\psi_k(t-)=m_kv_k(t-)+m_k\sum_{j=1}^N m_j\omega(x_k(t)-x_j(t)).
    \end{equation}
  Then we consider $x_j(t)$ in place of $x_j(t-)$ for any $j\in\{1,\dots,N\}$ by the fact that the trajectories of particles are continuous.
    We consider the sum over $k$ in the set $J_i(t)$, what we obtain is
    \begin{align}\label{eq:1}
        &\sum_{k=i_*(t)}^{i^*(t)}m_k\psi_k(t-)=\sum_{k=i_*(t)}^{i^*(t)}\Bigr[m_kv_k(t-)+m_k\sum_{j=1}^N m_j\omega(x_k(t)-x_j(t))\Bigl].
    \end{align}
    On the other hand, the term $\displaystyle{\sum_{k=i_*(t)}^{i^*(t)}}m_k\psi_i(t)$ can be rewritten as follows,
    \begin{align}
        \notag
        \sum_{k=i_*(t)}^{i^*(t)}m_k\psi_i(t)&=\sum_{k=i_*(t)}^{i^*(t)}m_k\Bigr( v_i(t)+\sum_{j=1}^N m_j\omega(x_k(t)-x_j(t))\Bigl)\\
        \label{eq:2}
        &=\sum_{k=i_*(t)}^{i^*(t)}m_kv_i(t)+\sum_{k=i_*(t)}^{i^*(t)}m_k\sum_{j=1}^N m_j\omega(x_k(t)-x_j(t)).
    \end{align}
    Using \eqref{eq:1} and  \eqref{eq:2}, we conclude that  
    \begin{equation*}
        \sum_{k=i_*(t)}^{i^*(t)}m_kv_i(t)=\sum_{k=i_*(t)}^{i^*(t)} m_kv_k(t-).
    \end{equation*}
    Then, the common velocity is
    \begin{equation*}
        v_i(t)=\frac{\sum_{k=i_*(t)}^{i^*(t)}m_kv_k(t-)}{\sum_{k=i_*(t)}^{i^*(t)}m_k},
    \end{equation*}
    and hence \eqref{eq: conservation of rho-psi} follows.
\end{proof}
We remark that the property in ${(ii)}$ is analogous to \eqref{eq:cons-of-psi}, 
indeed we have the conservation of the quantity $\psi$ along the trajectories:
\begin{equation*}
    0=\partial_t\psi+u\partial_x\psi=\frac{d}{dt}\psi\,.
\end{equation*}

\begin{prop}\label{prop2.2}
    Suppose $(x_i(t),v_i(t))_{i=1}^N$ follow the sticky particle Cucker-Smale dynamics 
    \eqref{eq:C-S omega}-\eqref{def:common-velocity} associated to the data $(x_i^0,v_i^0,m_i)_{i=1}^N$ 
    with the masses such that $\displaystyle{\sum_{i=1}^N} m_i=1$. Then, if we define 
    $$
    \psi_i^0:=v_i^0+\sum_{j=1}^N m_j \omega(x_i^0-x_j^0),\qquad i=1,\dots,N,
    $$
    and
    $$
    \underline{\psi}:=\min_{1\leq j\leq N}\psi_j^0,\qquad \bar{\psi}:=\max_{1\leq j\leq N}\psi_j^0,
    $$
    it holds
    \begin{equation}\label{eq:boundnedness psi}
         \psi_i(t)\in [\underline\psi,\bar\psi],\quad\forall i=1,\dots,N,\quad\forall t\geq0.
    \end{equation}
    Moreover,
    \begin{equation}\label{eq:boundedness v}
        \underline\psi-\|\omega\|_{L^\infty(\R)}\leq v_i(t)\leq \bar\psi,\quad\forall i=1,\dots,N,\quad t\geq0.
    \end{equation}
    Consequently, if $\{x_i^0\}_{i=1}^N\subset [-R^0,R^0]$, then
    \begin{equation}\label{eq:boundedness x}
        \{x_i(t)\}_{i=1}^N\subset[-R(t),R(t)],\quad R(t):=R^0+t\tilde M,\quad\forall t\geq0,
    \end{equation}
    where 
    $$
    \tilde M:=\max\{\,|\underline\psi-\|\omega\|_{L^\infty(\R)}|,\,|\bar\psi|\,\}.
    $$
\end{prop}

\begin{proof}
    We start with the proof of \eqref{eq:boundnedness psi}. By the conservation of  $\psi_i$ 
    given by \eqref{eq: conservation of rho-psi}
    \begin{itemize}
    	\item  when $t$ is a collision time we can conclude that
    \begin{equation*}
        \psi_i(t+)\geq\min_{1\leq j\leq N} \psi_j(t-);
    \end{equation*}
    \item when $t$ is not a collision time, by the fact that $\frac{d}{dt}\psi_i(t)=0$ we can conclude that
    \begin{equation*}
        \psi_i(t)\geq\underline\psi,\quad\forall i\in\{1,\dots,N\},\quad\forall t\geq0.
    \end{equation*}
    \end{itemize}
   Applying a similar argument we end up with 
    \begin{equation*}
        \psi_i(t)\leq\bar\psi,\quad\forall i\in\{1,\dots,N\},\quad\forall t\geq0,
    \end{equation*}
    thus, \eqref{eq:boundnedness psi} is proved.\newline
    Now we prove \eqref{eq:boundedness v}. 
    From the definition of $\psi_i$ in \eqref{def:psi-i}, we can write
    $$
    v_i(t)=\psi_i(t)-\sum_{j=1}^Nm_j\omega(x_i(t)-x_j(t)),
    $$
    and notice that
    \begin{equation}\label{eq:bound_above}
    \psi_i(t)-\sum_{j=1}^Nm_j\omega(x_i(t)-x_j(t))\leq\bar\psi,
    \end{equation}
    due to $\omega(x)\geq0$ for all $x\in\R$.
    Moreover, thanks to \eqref{eq:mass_1} we can write 
    \begin{equation}\label{eq:estimate_below}
    \psi_i(t)-\sum_{j=1}^Nm_j\omega(x_i(t)-x_j(t))\geq \underline\psi-\|\omega\|_{L^\infty(\R)}.
    \end{equation}
    Using \eqref{eq:bound_above}-\eqref{eq:estimate_below} it is possible to see that the absolute values of the velocities are always less or equal then 
    \begin{equation*}
        \tilde M:=\max\{\,|\underline\psi-\|\omega\|_{L^\infty(\R)}|,|\bar\psi|\}.
    \end{equation*}
    Defining $R(t):=R^0+t\tilde M$, \eqref{eq:boundedness x} holds true.
\end{proof}

The last property to prove about the sticky particle system is the following lemma.
\begin{lemma}[\textbf{The barycentric lemma}]\label{lemma: barycentric}
    Fix an $i\in\{1,\dots,N\}$ and a time $t>0$. For any $k\in J_i(t)$ the following holds
    \begin{equation}\label{barycentric equation}
        \frac{\sum_{j=i_*(t)}^km_j\psi_j(t-)}{\sum_{j=i_*(t)}^km_j}\geq
                \frac{\sum_{j\in J_i(t)}m_j\psi_j(t-)} {\sum_
          {{\color{black}j\in J_i(t)}} m_j}=\psi_i(t+)\geq\frac{\sum_{j=k}^{i^*(t)}m_j\psi_j(t-)}{\sum_{j=k}^{i^*(t)}m_j}
    \end{equation}
    {\color{black} where $i_*(t)$, $i^*(t)$ are defined in \eqref{def:i-star}.}
\end{lemma}

\begin{proof}
To prove the lemma it is sufficient to establish the following monotonicity property
    \begin{equation}\label{monotonicity of psi}
        \psi_{i_*(t)}(t-)\geq\dots\geq\psi_{i^*(t)}(t-).
    \end{equation}
     Let us show how we obtain \eqref{barycentric equation} from \eqref{monotonicity of psi}.\newline
    Without loss of generality we suppose to have some quantities $b_i$ and $\alpha_i$ for $i=1,\dots,M$ such that
    \begin{equation}\label{eq:monotonicity_b_i}
       b_1\ge b_2\ge\dots\ge b_M, \quad\sum_{i=1}^M\alpha_i=C,\quad C\in(0,1]\,.
    \end{equation}   
    We only prove the first inequality in \eqref{barycentric equation}, the proof is similar for the second one. 
    We consider a $\bar k\in\{1,\dots,M\}$ and the aim is to show that
    \begin{equation*}
        \frac{\sum_{j=1}^{\bar k}\alpha_jb_j}{\sum_{j=1}^{\bar k}\alpha_j}\geq\frac{\sum_{j=1}^M\alpha_jb_j}
        {\sum_{j=1}^M\alpha_j}=\frac{\sum_{j=1}^M\alpha_jb_j}{C},
    \end{equation*}
    that is equivalent to
    \begin{equation*}
        C\sum_{j=1}^{\bar k}\alpha_jb_j\geq\left(\sum_{j=1}^{\bar k}\alpha_j\right)\left(\sum_{j=1}^M\alpha_jb_j\right)
        =\left(\sum_{j=1}^{\bar k}\alpha_j\right)\left(\sum_{j=1}^{\bar k}\alpha_jb_j+\sum_{j=\bar k+1}^M\alpha_jb_j\right),
    \end{equation*}
    and this is true if and only if the following inequality holds true
    \begin{equation*}
        \left(C-\sum_{j=1}^{\bar k}\alpha_j\right)\left(\sum_{j=1}^{\bar k}\alpha_jb_j \right)\geq 
        \left(\sum_{j=1}^{\bar k}\alpha_j\right)\left(\sum_{j=\bar k+1}^M\alpha_jb_j\right).
    \end{equation*}
    We know that $\sum_{j=1}^M\alpha_j=C,$ then putting it in the last inequality we obtain the following
    \begin{equation*}
        \left(\sum_{j=1}^{\bar k}\alpha_jb_j\right)\left(\sum_{j=\bar k+1}^M\alpha_j\right)
        \geq\left(\sum_{j=1}^{\bar k}\alpha_j\right)\left(\sum_{j=\bar k+1}^M\alpha_jb_j\right).
    \end{equation*}
    Writing explicitly the sums, we have
    \begin{equation*}
        \left(\alpha_1b_1+\dots+\alpha_{\bar k}b_{\bar k}\right)\left(\sum_{j=\bar k+1}^M\alpha_j\right)
        \geq\left(\alpha_1+\dots+\alpha_{\bar k}\right)\left(\sum_{j=\bar k+1}^M\alpha_jb_j\right).
    \end{equation*}
    From \eqref{eq:monotonicity_b_i}  we obtain that the inequality holds true. 
    Substituting the $b_i$'s and the $\alpha_i$'s respectively with the $\psi_i(t-)$'s and the $m_i$'s, 
    we get \eqref{barycentric equation} if \eqref{monotonicity of psi} holds true.
    \newline
    It remains to prove \eqref{monotonicity of psi}, i.e. at a collision time $t$ one has
 \begin{equation*}
        \psi_{k}(t-)\geq\psi_{k+1}(t-)\qquad \forall\, k = i_*(t), \ldots, i^*(t)-1\,.
    \end{equation*}
    We know that for any $j\in J_i(t)$ we have
    $$
    \psi_k(t-)=v_k(t-)+\sum_{j=1}^N m_j\omega(x_k(t)-x_j(t)).
    $$
    Then, it is possible to conclude that \eqref{monotonicity of psi} follows 
    from the obvious monotonicity property of the velocities, i.e.
    \begin{equation*}
        v_{i_*(t)}(t-)\geq\dots\geq v_{i^*(t)}(t-),
    \end{equation*}
   that is true due to the fact that $t$ is a collision time. This completes the proof.
\end{proof}

\section{The scalar balance law}\label{Sec:3}
\setcounter{equation}{0}
Inspired by the work of Brenier and Grenier \cite{Brenier-Grenier} and the work of Leslie and Tan \cite{Leslie-Tan}, 
we prove the well-posedness of a scalar balance law linked to \eqref{ARZ_conservative}, 
and we use its entropy condition as selection principle for the uniqueness of the solution to \eqref{ARZ_conservative}. 
To do this, we define two new quantities, which are respectively the antiderivatives of $\rho$ and $\rho\psi$,
\begin{equation}\label{def-M-Q}
    M(x,t):=\int_{-\infty}^x \rho(y,t)\,dy,\qquad Q(x,t):=\int_{-\infty}^x (\rho\psi)(y,t)\,dy.
\end{equation}
Integrating \eqref{ARZ_conservative} in the space variable, we verify that $M$ and $Q$ satisfy 
\begin{equation}\label{eq-M-Q}
    \partial_t M+u\partial_xM=0,\qquad \partial_t Q+u\partial_x Q=0.
\end{equation}
We observe that $M$ is a monotone function, being $\rho(x,t)$ \textcolor{black}{non negative,} \textcolor{black}{and so choosing an initial data}
\begin{equation}\label{eq:initial-data-M} 
(M,Q)(x,0)=(M^0, Q^0),
\end{equation}
\textcolor{black}{such that $Q^0$ is constant in the intervals where $M^0$ is also constant, it is possible to define a map $A$ such that
such that $A(M^0(x))=Q^0(x)$ for any $x\in\R$, moreover 
if $(x,t)$ lies on a characteristic path 
generated by $(x_0,0)$ with velocity field $u$, then
\begin{equation*}
    Q(x,t)=Q^0(x_0)=A(M^0(x_0))=A(M(x,t)).
\end{equation*}
}
Using this relation and the definition of $\psi$, we get
\begin{align*}
    u\partial_x M &= (\psi-\omega\ast\rho)\rho=\psi\rho-(\omega\ast\rho)\rho \\
    &=\partial_x(A(M))-(\omega\ast\partial_xM)\partial_xM=\partial_x(A(M))-(\omega'\ast M)\partial_x M \\
    &=\partial_x(A(M))-(\phi\ast M)\partial_x M.
\end{align*}
Substituting this last equation in \eqref{eq-M-Q}$_1$, 
we obtain a scalar balance law for $M$,
\begin{equation}\label{scalar balance law}
    \partial_tM+\partial_x(A(M))=(\phi\ast M)\partial_x M,
\end{equation}
where we use 
$
\phi\ast M=\omega\ast\rho.
$
Indeed we have
\begin{align*}
\phi\ast M(x,t)&=\int_{-\infty}^{\infty}\phi(y)M(x-y,t)\,dy\\
&=\omega(y)M(x-y)|_{-\infty}^{\infty}+\int_{-\infty}^{\infty}\omega(y)\rho(x-y,t)\,dy=\omega\ast\rho(x,t).
\end{align*}
We define an entropy/entropy flux pair $(\eta, q)$ as a pair of functions $\eta=\eta(M)$ and $q=q(M)$ such that
$$
\eta:[0,1]\to\R,\qquad  q:[0,1]\to\R,\qquad q'=\eta'A',
$$
with $\eta$ Lipschitz and convex. The entropy inequality is
\begin{equation}\label{eq:entropy_inequality}
\int_0^\infty\int_\R \partial_t\varphi(\eta(M))+\partial_x\varphi(q(M))\,dx\,dt+\int_0^T\int_\R\varphi(\phi\ast M)\,d\partial_x(\eta(M))\,dt\geq0,
\end{equation}
where $\varphi\in C_c^\infty(\R\times[0,\infty))$ is a test function.
\begin{rem}
The first integral of the inequality can be interpreted in the sense of distribution, while in the second term $\partial_x(\eta(M))$ is a Radon measure. 
Thanks to \cite[Th. 4.19 (Young)]{Cannarsa-D'Aprile} we know that the term $\phi\ast M$ is in $L^\infty(\R)\cap C^0(\R)$, then
the integral
\begin{equation}\label{def:int-eta}
\int_0^\infty\int_\R \varphi(\phi\ast M)\,d(\partial_x(\eta(M))\,dt
\end{equation}
is well defined by the fact that $\partial_x(\eta(M))$ is a Radon measure, $\phi\ast M\in L^\infty(\R)\cap C^0(\R)$ and $\varphi\in C_c^\infty(\R)$,  thus the product $\varphi(\phi\ast M)$ has compact support.
\end{rem}
\begin{theorem}
    Let $\Omega\subset\R^N$ be an open set. Let $u\in BV_{loc}(\Omega)$, then there exists a unique 
    $\R^N$-valued Radon measure $\mu$ such that 
    $$
    \int_\Omega u \Div  \tilde\varphi\,dx=-\int_\Omega \tilde\varphi\,d\mu,\qquad \forall \tilde\varphi\in C_c^1(\Omega).
    $$
\end{theorem}
Thanks to this result we rewrite \eqref{def:int-eta} as
$$
\int_0^\infty\int_\R \varphi(\phi\ast M)\partial_x(\eta(M))\,dx\,dt,
$$
and so we can consider \eqref{eq:entropy_inequality} in distributional sense. We provide the definition of entropy weak solution to \eqref{eq:initial-data-M}-\eqref{scalar balance law}.
\begin{definition}\label{def:entropy_sol_M}
    Let $M:\R\times[0,T]\to\R$, with $T>0$, we say that $M$ is an entropy weak solution of \eqref{scalar balance law}-\eqref{eq:initial-data-M} 
    if it satisfies the following conditions:
    \begin{itemize}
        \item \textcolor{black}{$(M,Q)(x,0)=(M^0,Q^0)$;}
        \item the inequality \eqref{eq:entropy_inequality} is satisfied for any entropy/entropy-flux pair $(\eta,q)$ in the sense discussed above;
        \item $M(\cdot,t)$ is a nondecreasing function for $t\in[0,T]$;
        \item there exists $R(T)>0$ such that $M(x,t)=0$ if $x<-R(T)$ and $M(x,t)=1$ if $x>R(T)$ for any $t\in[0,T]$.
    \end{itemize}
\end{definition}
It is possible to extend Definition \ref{def:entropy_sol_M} of entropy weak solution to $\R\times\R_+$, 
requiring that for any constant $\bar T>0,$ $M(x,t)$ is an entropy weak solution in $\R\times[0,\bar T]$.
\newline
Now, we state the Rankine-Hugoniot condition and the Oleinik entropy condition. Assume that $M$ takes the values 
$M_l$ and $M_r$ on the two sides of a shock along a curve $C=\{(x,t):x=x(t)\}$, and denote by $s(t)=\dot x(t)$ the shock-speed. 
The entropy condition becomes
\begin{equation}\label{entropy_cond1}
\int_C ([\eta(M)]n_2(s)+([q(M)]-\phi\ast M)[\eta(M)])n_1(s))\varphi\,ds\geq0,
\end{equation}
where $n=(n_1,n_2)=\frac{1}{\sqrt{1+s^2}}(1,-s)$, with the notation $[f]=f(M_l)-f(M_r)$. 
Substituting $n_1$ and $n_2$ in \eqref{entropy_cond1}, we get
\begin{equation}\label{eq:Rh_1}
(s+(\phi\ast M))[\eta(M)]\leq[q(M)],
\end{equation}
along the curve $C$. Choosing $\eta(x)=x$ and $q(M)=A(M)$, the inequality \eqref{eq:Rh_1} becomes an equality, 
then we have the following Rankine-Hugoniot condition
\begin{equation}\label{eq:RH_M}
s+\phi\ast M=\frac{[A(M)]}{[M]}.
\end{equation}
To get the Oleinik entropy condition, we consider the following.
If $M_l<M_r$, we take $\eta(m)=(m-\theta)H(m-\theta)$ and $q(m)=(A(m)-A(\theta))H(m-\theta)$ with $\theta\in(M_l,M_r),$ 
where the function $H$ is the right-continuous Heaviside function
\begin{equation*}
    H(x)=
    \begin{cases}
        1,\quad x\geq 0,\\
        0,\quad x<0.
    \end{cases}
\end{equation*}
Then, we get
\begin{equation}\label{eq:oleinik}
s+\phi\ast M\leq\frac{A(\theta)-A(M_l)}{\theta-M_l}.
\end{equation}

\subsection{\color{black}Construction of approximated solutions%
\label{sec:scalar_balance_law}
}
\textcolor{black}{In this subsection we investigate the existence of entropy weak solutions to the scalar balance law \eqref{scalar balance law} and 
we start by studying the following discretized formulation:}
\begin{equation}\label{eq:discretized scalar balance law}
    \partial_t M_N+\partial_x(A_N(M_N))=(\phi\ast M_N)\partial_x M_N 
\end{equation} 
together with initial datum
\begin{equation}\label{eq:discretized scalar balance law-init data}
   M_N(x,0)=M_N^0(x)\,.
\end{equation}
Here the flux function $A_N$ is a continuous and piecewise linear function with breakpoints at $\theta_i$ for $i=1,\dots,N-1$:
\begin{equation}\label{def: A_N}
    A_N:[0,1]\rightarrow \R, \quad \mbox{$A_N'$ is constant in each interval $(\theta_{i-1},\theta_i$)}, \quad
    i=1,\dots,N.
\end{equation}
\textcolor{black}{The main purpose of this subsection is to construct 
an entropy weak solution to 
\eqref{eq:discretized scalar balance law}-\eqref{eq:discretized scalar balance law-init data} 
by means of 
the sticky particle Cucker-Smale dynamics.}

To start, let's consider a piecewise constant initial datum $M_N^0$ as follows:
\begin{equation}\label{def: M_N^0}
    M_N^0(x)=\sum_{j=1}^N m_jH(x-x_j^0),
\end{equation}
where $x_1^0\leq x_2^0\leq ...\leq x^0_N$ are the initial positions of particles in the sticky particle 
Cucker-Smale dynamics \eqref{eq:C-S}-\eqref{eq:C-S omega}-\eqref{def:common-velocity}, 
and the masses are such that $m_j>0$ 
for any $j=1,\dots,N$ and $\displaystyle{\sum_{j=1}^N} m_j=1$. 
\newline
It is worth noticing that by construction $M_N^0$ assume a finite number of values, indeed the range of $M_N^0$ 
is discrete and given by the values $\theta_i,\, i\in\{0,...,N\}$, defined in this way:
\begin{equation}\label{def: theta_i}
   \theta_0=0\,,\qquad \theta_i= \theta_{i-1}+ m_i = \sum_{j=1}^i m_j\,,\qquad i=1,\ldots,N.
\end{equation}
Later we will see that also the solution $M_N$ assumes a finite number of values.
\newline
Under these hypotheses, \textcolor{black}{we prove the existence of an entropy weak solution to \eqref{eq:discretized scalar balance law}-\eqref{eq:discretized scalar balance law-init data} with initial data as in \eqref{def: M_N^0}}.

\begin{theorem}\label{th:solution_MN}
\textcolor{black}{Let $A_N$ be a flux function satisfying \eqref{def: A_N}, and consider $M_N^0$ as in \eqref{def: M_N^0}. 
Define}
\begin{equation}\label{eq: psi_i^0}
    m_i\psi_i^0:=A_N(\theta_i)-A_N(\theta_{i-1})\,,\qquad i=1,\dots,N
\end{equation}
and
\begin{equation}
    v_i^0:=\psi_i^0-\sum_{j=1}^N m_j\omega(x_i^0-x_j^0).
\end{equation}
Suppose that the couples $(x_i(t),v_i(t))_{i=1}^N$ follow the sticky particle Cucker-Smale dynamics 
\eqref{eq:C-S}-\eqref{eq:C-S omega}-\eqref{def:common-velocity} with masses $(m_i)_{i=1}^N$ and initial conditions $(x_i^0,v_i^0)_{i=1}^N$. 
Then, the function
\begin{equation}\label{def: M_N solution}
    M_N(x,t)=\sum_{i=1}^Nm_iH(x-x_i(t))
\end{equation}
is an entropy weak solution of the discretized Cauchy problem \eqref{eq:discretized scalar balance law}-\eqref{eq:discretized scalar balance law-init data} 
satisfying the following:
\begin{itemize}
    \item[\textbf{(a)}] $M_N(x,t)=0$ for $x<-R(T)$ and $M_N(x,t)=1$ for $x>R(T)$ and $\forall t\in [0,\infty)$;
    \item[\textbf{(b)}] it holds true that
\begin{equation}\label{eq: A_N(M_N)}
    A_N\circ M_N(x,t)=\sum_{i=1}^Nm_i\psi_i(t)H(x-x_i(t));
\end{equation}
\item[\textbf{(c)}] at points of discontinuity $x_i(t)$, $i=1,\dots,N$, we have 
\begin{equation}\label{eq:eq_discontinuity}
        v_i(t)+\phi\ast M_N(x_i(t),t)=\frac{A_N(\theta_{i^*(t)})-A_N(\theta_{i_*(t)-1})}{\theta_{i^*(t)}-\theta_{i_*(t)-1}},
    \end{equation}
namely, the Rankine-Hugoniot condition is satisfied;
 \item[\textbf{(d)}] at the points of discontinuity the following inequality holds true
\begin{equation}\label{eq:ineq_discontinuity}
        v_i(t)+\phi\ast M_N(x_i(t),t)\leq\frac{A_N(\theta)-A_N(\theta_{i_*(t)-1})}{\theta-\theta_{i_*(t)-1}},
    \end{equation}
with $\theta\in(\theta_{i_*(t)-1},\theta_{i^*(t)})$. Thus the Oleinik entropy condition is satisfied.
\end{itemize}
\end{theorem}
\begin{proof}
\textcolor{black}{
We show that the function $M_N(x,t)$ defined in \eqref{def: M_N solution} is an entropy weak solution to \eqref{eq:discretized scalar balance law}-\eqref{eq:discretized scalar balance law-init data}, namely, we prove that it satisfies points \textbf{(a), (b), (c)}, and \textbf{(d)} of the statement. The proof is divided into steps for a better readability.
}
\begin{itemize}   
    \item[\textbf{(a)}] This point holds true by construction, indeed
    $$
    M_N(x,t)=\sum_{i=1}^N m_i H(x-x_i(t)),
    $$
    and, by Proposition \ref{prop2.2}, we have $\{x_i(t)\}_{i=1}^N\subset [-R(T),R(T)],$ then \textbf{(a)} holds true.
    \item[\textbf{(b)}]
    The case $x<x_1(t)$ is trivial.
    For any $x\geq x_1(t)$,  let us consider the highest index $k\in\{1,\dots,N\}$, such that $x_k(t)\leq x<x_{k+1}(t)$, \textcolor{black}{with $x_{N+1}(t)=+\infty$,} then we have 
    \begin{equation}\label{eq:An_Mn}
        A_N(M_N(x,t))=\sum_{j=1}^k (A_N(\theta_i)-A_N(\theta_{i-1}))=\sum_{j=1}^k m_j\psi_i^0,
    \end{equation}
    moreover, from \eqref{eq: conservation of rho-psi} we get
    \begin{equation}\label{eq:sum_conserv}
        \sum_{i=1}^k m_i\psi_i^0=\sum_{i=1}^k m_i\psi_i(t)=\sum_{i=1}^N m_i\psi_i(t)H(x-x_i(t)).
    \end{equation}
    Substituting \eqref{eq:sum_conserv} in \eqref{eq:An_Mn}, we obtain (\ref{eq: A_N(M_N)}).
    \item[\textbf{(c)}] 
    We can see by construction that the function $M_N(\cdot,t)$ is piecewise constant with discontinuities only along the curves $C_i=\{(x_i(t),t)\,:\,t\geq0\}$. 
    Furthermore, we observe that $M_N(x,t)$ takes values in the set $\{0,\theta_1,\ldots,\theta_{N-1},1\},$
    therefore $A_N(M_N (x,t))$ is well defined outside the curves $C_i$.
    Thus, we need to verify the conditions along these discontinuities with shock speed $\sigma_i(t)=v_i(t)$. 
    Let us fix a time $t$ and a point 
    $(x_i(t),t)$ on the curve $C_i$. By definition of $M_N$ in \eqref{def: M_N solution}, we obtain
    \begin{align*}
        M_N(x_i(t)-,t)&=\theta_{i_*(t)-1},\\
        M_N(x_i(t)+,t)&=M_N(x_i(t),t)=\theta_{i^*(t)},
    \end{align*}
    then,
    \begin{equation}\label{eq:cumulative-mass}
        M_N(x_i(t)+,t)-M_N(x_i(t)-,t)=\theta_{i^*(t)}-\theta_{i_*(t)-1}=\sum_{j\in J_i(t)}m_j.
    \end{equation}
    Performing the same computations for $A_N$, we get
    \begin{align} \nonumber
        A_N\circ M_N(x_i(t)+,t)&-A_N\circ M_N(x_i(t)-,t)=A_N(\theta_{i^*(t)})-A_N(\theta_{i_*(t)-1})\\
 \nonumber
        &=A_N(\theta_{i^*(t)})\pm\left(\sum_{j=i_*(t)}^{i^*(t)-1}A_N(\theta_j)\right)-A_N(\theta_{i_*(t)-1})\\ \nonumber
        &=\sum_{j\in J_i(t)}m_j\psi_j^0\\&=\psi_i(t) \sum_{j\in J_i(t)}m_j,\label{eq:01}
    \end{align}
where the last two equalities are obtained using the relations (\ref{eq: psi_i^0}) and by iterating \eqref{eq: conservation of rho-psi} across interactions. 
Thanks to \cite[Th. 4.19 (Young)]{Cannarsa-D'Aprile}, the function $\phi\ast M_N\in L^\infty(\R)\cap C^0(\R)$ and {\blu can be expressed as follows:}
\begin{align*}
\phi\ast M_N(x,t) & =  \sum_{j=1}^N  \int_\R \phi(x-y) m_j H(y-x_j(t)) \, dy\\
&=   \sum_{j=1}^N  m_j \int_{x_j(t)}^{+\infty} \phi(x-y)  \, dy
=   \sum_{j=1}^N  m_j \int_{-\infty}^{x-x_j(t)} \phi(\xi)  \, d\xi
\\
&=   \sum_{j=1}^N  m_j \,\omega( x - x_j(t)).
\end{align*}
In particular, the pointwise value of $\phi\ast M_N$ is well defined as it is
\begin{equation}\label{eq:M_N-xi}
    \phi\ast M_N(x_i(t),t) = \sum_{j=1}^Nm_j\omega(x_i(t)-x_j(t)).
\end{equation}
 Putting together \eqref{eq:cumulative-mass}, \eqref{eq:01}, and using \eqref{eq:M_N-xi}, we obtain
    \begin{align*}
         &\frac{A_N\circ M_N(x_i(t)+,t)-A_N\circ M_N(x_i(t)-,t)}{M_N(x_i(t)+,t)-M_N(x_i(t)-,t)}=\frac{\sum_{j\in J_i(t)}m_j\psi_j^0}{\sum_{j\in J_i(t)}m_j}\\
         &=\psi_i(t)
         =v_i(t)+\sum_{j=1}^Nm_j\omega(x_i(t)-x_j(t))       
         \\&=v_i(t)+\phi\ast M_N(x_i(t),t),
    \end{align*}
    then the Rankine-Hugoniot condition 
    is proved. 
    \item[\textbf{(d)}] 
    Now, we check \eqref{eq:ineq_discontinuity}: we have 
    \begin{equation}\label{eq:oleinik_1}
        v_i(t)+\phi\ast M_N(x_i(t),t)\leq\frac{A_N(\theta)-A_N(\theta_{i_*(t)-1})}{\theta-\theta_{i_*(t)-1}},
    \end{equation}
    with $\theta\in(\theta_{i_*(t)-1},\theta_{i^*(t)})$.
    Being $A_N$ piecewise linear, it is sufficient to prove that \eqref{eq:oleinik_1} is verified at the breakpoints $\theta=\theta_k$ with $k\in\{i_*(t),\dots,i^*(t)-1\}$. Then, we can write 
    \begin{align*}
        \frac{A_N(\theta_k)-A_N(\theta_{i_*(t)-1})}{\theta_k-\theta_{i_*(t)-1}}&
        =\frac{\sum_{j=i_*(t)}^k m_j\psi_i^0}{\sum_{j=i_*(t)-1}^k m_j}\geq\frac{\sum_{j\in J_i(t)}m_j\psi_i^0}{\sum_{j\in J_i(t)} m_j}
        \\&=\psi_i(t)=v_i(t)+\phi\ast M_N(x_i(t),t),
    \end{align*}
    where the inequality is due to Lemma \ref{lemma: barycentric}.
    \end{itemize}
\end{proof}
\subsection{\color{black}The solution of the scalar balance law}
The next step is to study the solution of the scalar balance law \eqref{scalar balance law} using the result obtained 
in Theorem \ref{th:solution_MN}. 
The idea is to pass to the limit as $N$ goes to $+\infty$. 
Let us consider the 
Cauchy problem
\begin{equation}\label{def:Cauchy_problem}
\begin{cases}
    \partial_tM+\partial_x(A(M))=(\phi\ast M)\partial_x M,\qquad \\M(x,0)=M^0(x), 
    \end{cases}
\end{equation}
for some given initial datum $M^0$. We have the following result.
\begin{theorem}\label{theorem:entropy-solution}
    Assume \textbf{(H)} and \eqref{phi-def}. Consider the Cauchy problem \eqref{def:Cauchy_problem} with $M^0$ a nondecreasing function, and
    suppose that there exists an $R^0>0$ such that
    $$
    M^0(-x)=0,\qquad M^0(x)=1,\qquad\text{if $x\geq R^0$}.
    $$
    Let the flux $A:[0,1]\to\R$ be a Lipschitz continuous function. Then we have the following results.
    \begin{itemize}
        \item[\textbf{(a)}] Given any $T>0$, the Cauchy problem \eqref{def:Cauchy_problem} has a unique entropy weak solution
        $$
        M\in BV(\R\times[0,T]).
        $$
        \item[\textbf{(b)}] For any $T>0$, the entropy weak solution $M$ can be approximated by the discretized balance law \eqref{eq:discretized scalar balance law} 
        with initial datum \eqref{eq:discretized scalar balance law-init data}, and thus by the sticky particle Cucker-Smale dynamics \eqref{eq:C-S}, 
        i.e. there exists a sequence of initial data $M_N^0$ and fluxes $A_N$ satisfying the assumptions \eqref{def: M_N^0}-\eqref{def: A_N} respectively, 
        such that the associated entropy weak solution $M_N$ constructed in Theorem \ref{th:solution_MN} satisfies
        \begin{equation}\label{eq:M_N-to-M}
            M_N-M\to0\quad\text{in}\quad C([0,T];L^1(\R)),
        \end{equation}
        and
        \begin{equation}\label{eq: dev-M_N-to-M}
            \partial_t M_N(\cdot,t)\overset{\ast}{\rightharpoonup} \partial_t M(\cdot,t)\quad\text{in}\quad \mathcal{M}(\R),
        \end{equation}
        for any $t\in[0,T]$, where $\mathcal{M}(\R)$ is the signed measure space.
        \item[\textbf{(c)}] Let $\tilde M$ be the entropy weak solution of 
        $$\begin{cases}
        \partial_t\tilde M+\partial_x(\tilde A(\tilde M))=(\phi\ast \tilde M)\partial_x\tilde M\quad\\\tilde M(x,0)=\tilde M^0,
        \end{cases}
        $$
        where the initial datum $\tilde M^0$ and the flux $\tilde A$ satisfy the same assumption as $M^0$ and $A$ respectively. 
        Then, for any $t\geq0$, the following stability bounds are satisfied
        \begin{equation}\label{eq:stability-estimate}
            \|M(\cdot,t)-\tilde M(\cdot,t)\|_{L^1(\R)}\leq e^{2t\|\phi\|_{L^\infty(\R)}}\|M^0-\tilde M^0\|_{L^1(\R)}+\frac{|A-\tilde A|_{Lip}}{2\|\phi\|_{L^\infty(\R)}}(e^{t\|\phi\|_{L^\infty(\R)}}-1),
        \end{equation}
        and
        \begin{equation}\label{eq:stability-estimate-two}
            \|M(\cdot,t)-\tilde M(\cdot,t)\|_{L^1(\R)}\leq \|M^0-\tilde M^0\|_{L^1(\R)}+(|A-\tilde A|_{Lip}+4\|\omega\|_{L^\infty(\R)})t.
        \end{equation}
    \end{itemize}
\end{theorem}

    \begin{proof}
        Let $N\in\N$ be fixed, we construct two functions $M_N$ and $A_N$ 
        satisfying the assumptions of Theorem \ref{th:solution_MN}. 
        Let us consider positive masses $m_{i,N}$ such that 
        \begin{equation}\label{def:hypotheses_m_iN}
            \sum_{i=1}^N m_{i,N}=1,\qquad \lim_{N\to\infty}\max_{1\leq i\leq N} m_{i,N}=0,
        \end{equation}
         and define the quantities $\theta_i$ as
        $$
        \theta_0:=0,\qquad \theta_i=\sum_{j=1}^i m_{i,N},\qquad \text{for $i=1,\dots,N$}.
        $$
        We define the quantities $x_{i,N}^0=x_{i,N}(0)$
        $$
        x_{i,N}^0:=\inf\{x: M^0(x)\geq\theta_i\},\qquad i=1,\dots,N.
        $$
       Notice that $x_{i,N}^0 \in [-R^0,R^0],\, \forall i\in\{1,...,N\}$. 
       Now, we can construct the functions $M_N^0$ and $A_N$ as in the following:
        $$
        M_N^0(x)=\sum_{j=1}^N m_{i,N}H(x-x_{i,N}^0),
        $$
        while $A_N$ is the piecewise linear approximation of $A$ such that
        \begin{equation}\label{def:A_N-theta}
            A_N(\theta_{i,N})=A(\theta_{i,N}),\qquad i=0,\dots,N.
        \end{equation}
        These functions satisfy the hypotheses \eqref{def: M_N^0} and \eqref{def: A_N} of Theorem \ref{th:solution_MN}, respectively, then
        $$
        M_N(x,t)=\sum_{i=1}^N m_{i,N}H(x-x_{i,N}(t))
        $$
        is an entropy weak solution to \eqref{eq:discretized scalar balance law}-\eqref{eq:discretized scalar balance law-init data}. We denote by $x_{i,N}(t)$, $i=1,\dots,N$ the trajectories of the particles involved in the Cucker-Smale dynamics 
        with initial data $(m_{i,N},x_{i,N}^0,v_{i,N}^0)_{i=1}^N$. 
        The initial velocities can be obtained using the relation 
        $$
        v_{i,N}^0=\psi_{i,N}^0+\sum_{j=1}^N m_{j,N}\omega(x-x_{i,N}^0),
        $$
        with $\psi_{i,N}^0$ such that 
        $$
        m_{i,N}\psi_{i,N}^0=A_N(\theta_{i,N})-A_N(\theta_{i-1,N}).
        $$
        Moreover, the quantities $M_N^0$ and $A_N$ approximate the quantities $M^0$ and $A,$ i.e., the following inequalities hold true
        \begin{align}
            \|M_N^0-M^0\|_{L^1(\R)}&\leq 2R^0\max_{1\leq i\leq N}m_{i,N},
            \label{eq:convergence-M^0}
            \\ \sup_{m\in[0,1]}|A_N(m)-A(m)|&\leq |A|_{Lip}\max_{1\leq i\leq N}m_{i,N},
            \label{eq:convergence-A}
        \end{align}
        and these imply that $M_N^0-M^0\to 0$ in $L^1(\R)$ and $A_N\to A$ uniformly as $N\to\infty$ using \eqref{def:hypotheses_m_iN}.\newline
        Let us prove \eqref{eq:convergence-M^0} and \eqref{eq:convergence-A}.
        \begin{itemize}
            \item To prove [\eqref{eq:convergence-M^0}], we denote $x_{0,N}^0:=-R^0$ and we write 
            \begin{align*}
                \|M_N^0&-M^0\|_{L^1(\R)}= \sum_{j=1}^N\int_{x_{j-1,N}}^{x_{j,N}}|M_N^0(x)-M^0(x)|\,dx\\
                &= \sum_{j=1}^N\int_{x_{j-1,N}}^{x_{j,N}}M^0(x)-M_N^0(x)\,dx \leq \sum_{j=1}^N m_{j,N}(x_{j,N}^0-x_{j-1,N}^0)\leq 2R^0\max_{1\leq j\leq N}m_{j,N}.
            \end{align*}
            \item To prove [\eqref{eq:convergence-A}], let fix $m$ in the set $[0,1]$. Then, there exists  an index $i$ such that $m\in[\theta_{i-1,N},\theta_{i,N})$. 
            Using \eqref{def: A_N} and \eqref{def:A_N-theta}, we obtain
            \begin{align*}
                &A_N(m)-A(m)\\
                &\qquad =\frac{\theta_{i,N}-m}{m_{i,N}}(A(\theta_{i-1,N})-A(m))+\frac{m-\theta_{i-1,N}}{m_{i,N}}(A(\theta_{i,N})-A(m))\\
                &\qquad \leq \frac{\theta_{i,N}-\theta_{i-1,N}}{m_{i,N}}|A(\theta_{i-1,N})-A(m)|+\frac{\theta_{i,N}-\theta_{i-1,N}}{m_{i,N}}|A(\theta_{i,N})-A(m)|\\
                &\qquad =|A(\theta_{i-1,N})-A(m)|+|A(\theta_{i,N}-A(m)|\leq |A|_{Lip}(|\theta_{i-1,N}-m|+|\theta_{i,N}-m|)\\
                &\qquad =|A|_{\Lip}(\theta_{i,N}-\theta_{i-1,N})=|A|_{\Lip}\,m_{i,N}\leq |A|_{\Lip}\max_{1\leq j\leq N} m_{j,N}.
            \end{align*}
        \end{itemize}
        Let $T>0$ be a fixed time, since $M_N(\cdot,t)$ is uniformly bounded and nondecreasing for any $t\in[0,T],$ we can apply the Helly's Theorem 
        stating that there exists a $L^1_{loc}(\R)-$convergent subsequence $M_{N_k}(\cdot, t)$. Applying a diagonal argument, it is possible to get 
        another subsequence convergent in $L^1_{loc}(\R)$ for all rational $t\in[0,T]$, we denote again this subsequence as $M_{N_k}$ to simplify the notation 
        and we denote with $M(\cdot,t)$ the limit of this subsequence. The aim is to extend this convergence from $L^1_{loc}(\R)$ to $L^1(\R)$ 
        and also to extend the result to irrational times. From 
        \eqref{eq:C-S omega}, 
        we know $\{x_{i,N}(t)\}_{i=1}^N\subset[-R(T),R(T)]$ with $R(T)=R^0+T\tilde M$, thus
        \begin{align*}
            M_N(x,t)&=0,\qquad x\leq -R(T),\\
            M_N(x,t)&=1,\qquad x\geq R(T),
        \end{align*}
        and so $M_{N_k}(t)-M(t)\to0$ in $L^1(\R)$ for all rational times $t\in[0,T]$.\newline
        The extension of this result of convergence to irrational times is a consequence of the following time regualar estimate
        \begin{align}
            \label{eq:regular-time-estimate}
            \int_\R |M_N(x,t)-M_N(x,s)|\,dx&\leq\sum_{i=1}^N m_{i,N}|x_{i,N}(t)-x_{i,N}(s)|\leq \max_{1\leq i\leq N}|v_{i,N}^0|(t-s)\\
            \nonumber
            &\leq \tilde M(t-s).
        \end{align}
        Combining \eqref{eq:regular-time-estimate} with the established convergence at rational times, we get
        \begin{equation}\label{eq:L1_convergence}
        M_{N_k}-M\to0\quad\text{in}\quad C([0,T];L^1(\R)).
        \end{equation}
        Moreover, by \eqref{eq:regular-time-estimate}, it holds
        $$
        \|\partial_t M_{N_k}(x,t)\|_{\mathcal{M}}\leq \tilde M,\qquad\forall t\in[0,T],
        $$
        then, we can extract a further subsequence, still denoted by $M_{N_k}$ for simplicity, such that
        $$
        \partial_t M_{N_k}(\cdot,t)\weakstar\partial_tM(\cdot,t)\quad\text{in}\quad \mathcal{M}(\R),
        $$
        and so we can conclude that $M\in BV(\R\times[0,T])$. 
        Now, we have to verify if the limit function $M$ satisfies the entropy condition \eqref{entropy_cond1}. 
        We check this using the Kru\v{z}kov entropy pairs $(\eta,q)$ 
        \textcolor{black}{
        \begin{equation}\label{def:Kruzkov}
          \eta(M)=|M-\alpha|,\quad  q(M)=\text{sgn}(M-\alpha)(A(M)-A(\alpha)).
        \end{equation}
        }
        We know that
        $$
        M_N(x,t)=\sum_{i=1}^N m_{i,N}H(x-x_{i,N}(t))
        $$
        is an entropy weak solution of the discretized scalar balance law \eqref{eq:discretized scalar balance law}; 
        this means that it satisfies the entropy inequality with the entropy/entropy-flux pair $(\eta,q_N)$, 
        where $\eta(m)=|m-\alpha|$, $q_N(m)=\sgn(m-\alpha)(A_N(m)-A_N(\alpha))$ and $\alpha\in[0,1]$
        \begin{equation}\label{eq:entropy_discretized}
            \int_0^T\int_\R [\eta(M_N)\partial_t\varphi+q_N(M_N)\partial_x\varphi+(\phi\ast M_N)\varphi\partial_x(\eta(M_N))]\,dx\,dt\geq0.
        \end{equation}
        Now, we consider the subsequence $M_{N_k}$ and we pass to the limit, to simplify the notation we use $M_N$ in place of 
        $M_{N_k}$. For a fixed $t\in[0,T]$ we obtain
        \begin{equation}\label{eq:M_N_M}
            \begin{aligned}
        \|\eta(M_N)-\eta(M)\|_{L^1(\R)}&=\int_\R |\eta(M_N)-\eta(M)|\,dx\leq |\eta|_{Lip}\int_\R |M_N-M|\,dx\\
        &=|\eta|_{Lip}\|M_N-M\|_{L^1(\R)}\to0,
        \end{aligned}
        \end{equation}
        thanks to the convergence result in \eqref{eq:L1_convergence}. Moreover,
        \begin{equation}\label{eq:q_N_q}
            \begin{aligned}
            \|q_N(M_N)&-q(M)\|_{L^1(\R)}  \leq \|q_N(M_N)-q(M_N)\|_{L^1(\R)}+\|q(M_N)-q(M)\|_{L^1(\R)}\\
            &\leq 2R(T)\|A_N-A\|_{L^\infty([0,1])}+|A|_{Lip}\|M_N-M\|_{L^1(\R)}\to0.
        \end{aligned}
        \end{equation}
         Thanks to \eqref{eq:M_N_M} and \eqref{eq:q_N_q}, the first two terms in \eqref{eq:entropy_discretized} converge to
\begin{equation*}
            \int_0^T\int_\R [\eta(M)\partial_t\varphi+q(M)\partial_x\varphi]\,dx\,.
\end{equation*}
About the last term, we start by showing that  $\phi\ast M_N$ converges to $\phi\ast M$ as $N\to +\infty$, 
indeed, since $\phi\in L^\infty(\R)$, we have
        $$
        \|\phi\ast M_N-\phi\ast M\|_{L^\infty(\R)} \leq \|\phi\|_{L^\infty(\R)}\|M_N-M\|_{L^1(\R)}\to 0.
        $$
       Let us consider the term $\partial_x\eta(M_N)$, thanks to \cite[Lemma 5.2]{Leslie-Tan} we know that
        $$
        \|\partial_x\eta(M_N)\|_{\mathcal{M}(\R)} \leq |\eta|_{Lip}
        $$
        and this is also true for $\partial_x(\eta(M))$. Now, we are able to prove the convergence of the last term 
        in 
        \eqref{eq:entropy_discretized}
        \begin{align*}
            \left|\int_0^T\int_\R(\phi\ast M_N)\right.&\left.\varphi\partial_x(\eta(M_N))\,dx\,dt-\int_0^T\int_\R (\phi\ast M)\varphi\partial_x(\eta(M))\,dx\,dt\right|\\
            \leq & \int_0^T \|\varphi(t)\|_{L^\infty(\R)}\|\phi\ast M_N-\phi\ast M\|_{L^\infty(\R)}\|\partial_x(\eta(M_N)\|_\MM\,dt \\
            &+\left| \int_0^T\int_\R (\phi\ast M)\varphi[\partial_x(\eta(M_N))-\partial_x(\eta(M))]\,dx\,dt\right|,
        \end{align*}
and since $\|\phi\ast M_N-\phi\ast M\|_{L^\infty(\R)}\to0$, the first term on the right-hand side above goes to zero.  
On the second term we have
        \begin{align*}
            &\left| \int_0^T\int_\R (\phi\ast M)\varphi[\partial_x(\eta(M_N))-\partial_x(\eta(M))]\,dx\,dt \right|\\
            &\qquad\qquad  = \left| \int_0^T\int_\R -\partial_x[(\phi\ast M)\varphi][(\eta(M_N))-(\eta(M))]\,dx\,dt \right|\\ 
            &\qquad\qquad  \leq \int_0^T\int_\R\abs{-\partial_x[(\phi\ast M)\varphi][(\eta(M_N))-(\eta(M))]}\,dx\,dt\\
            &\qquad\qquad  \leq \|\partial_x[(\phi\ast M)\varphi]\|_{L^\infty(\R)}\int_0^T\|\eta(M_N)-\eta(M)\|
            _{L^1(\R)}\,dt\,.
        \end{align*}
        Since $\eta$ is Lipschitz continuous and by means of \eqref{eq:L1_convergence}, the right-hand side above converges to zero and thus, 
        $M$ is an entropy weak solution of \eqref{def:Cauchy_problem}.
        
        The uniqueness of solutions is a direct consequence of the stability estimate \eqref{eq:stability-estimate} setting $\tilde M^0=M^0$ and $\tilde A=A$. 
        The argument of the proof of \eqref{eq:stability-estimate} is based on the Kru\v{z}kov's doubling of the variables strategy, see \cite{Kruzkov1970}. 
        For fixed $(y,s)$, consider the Kru\v{z}kov entropy pair \eqref{def:Kruzkov} with $\alpha=\tilde M(y,s)$ and a test function $\varphi(x,t)=w(x,t,y,s)$. 
        In the following the integration bounds are omitted to simplify the notation. The spatial variables $x$ and $y$ are integrated over $\R$ 
        and the time variables $t$ and $s$ over the set $[0,T]$.
        Thus, the entropy inequality \eqref{eq:entropy_inequality} becomes 
        \begin{align*}
            \iint &|M(x,t)-\tilde M(y,s)|\partial_t w(x,t,y,s)\,dx\,dt\\
            &+\iint \sgn(M(x,t)-\tilde M(y,s))(A(M(x,t))-A(\tilde M(y,s)))\partial_x w(x,t,y,s)\,dx\,dt\\
            &+\iint (\phi\ast M)(x,t)(\partial_x|M(x,t)-\tilde M(y,s)|)\,dx\,dt.
        \end{align*}
         We perform the same for the entropy inequality of $\tilde M$, but with $\tilde A$ in place of $A$ and with $\alpha=\tilde M(y,s)$. 
         We integrate both the inequalities above over the remaining free variables in both cases and adding them, we get
         \begin{equation} \label{eq:doubling-Kruzkov}
             \begin{aligned}
            0 \leq &\iiiint|M(x,t)-\tilde M(y,s)|(\partial_t w + \partial_s w)(x,t,y,s)\,dx\,dt\,dy\,ds\\
            & + \iiiint \sgn(M(x,t)-\tilde M(y,s))[(A(M(x,t))-A(\tilde M(y,s)))\partial_x w(x,t,y,s)\\
            & \hspace{10em}+ (\tilde A(M(x,t))-\tilde A(\tilde M(y,s)))\partial_y w(x,t,y,s)]\,dx\,dt\,dy\,ds\\
            & + \iiiint w(x,t,y,s)[(\phi\ast M)(x,t)\partial_x|M(x,t)-\tilde M(y,s)|\\
            & \hspace{10em}+ (\phi\ast \tilde M)(y,s)\partial_y|M(x,t)-\tilde M(y,s)|]\,dx\,dt\,dy\,ds.
        \end{aligned}
        \end{equation}
        We introduce the auxiliary variables
        $$
        \bar x= \frac{x+y}{2},\quad \bar y= \frac{x-y}{2},\quad \bar t= \frac{s+t}{2},\quad \bar s= \frac{t-s}{2},
        $$
        and we choose test function
        $$
        w(x,t,y,s)=b_\eps\left( \frac{x-y}{2}\right)b_\eps\left( \frac{t-s}{2}\right)g\left( \frac{x+y}{2} \right)h_\delta\left( \frac{t+s}{2} \right)
        =b_\eps(\bar y)b_\eps( \bar s)g(\bar x)h_\delta(\bar t).
        $$
        These test functions satisfy the following properties.
        \begin{itemize}
        \item $w(x,t,y,s)$ are smooth and nonnegative.
            \item The functions $(b_\eps)_{\eps>0}$ approximate the Dirac delta distribution as $\eps\to0+$, 
            we consider $b_\eps$ as a standard mollifier supported in $(-\eps,\eps)$ and such that the integral over $\R$ is equal to 1.
            \item The function $g$ is identically 1 on $[-R(T),R(T)]$ and it is compactly supported.
            \item The functions $(h_\delta)_{\delta>0}$ approximate the indicator function of $[s,t]$ as $\delta\to0+$. 
            We consider $h_\delta$ to be identically 1 on $[s,t]$, identically 0 outside $[s-\delta,t+\delta]$ and linear on $[s-\delta,s]$ and $[t,t+\delta]$.
        \end{itemize}
        To proceed, we substitute the test function in \eqref{eq:doubling-Kruzkov} using also the auxiliary variables and notice that
        $$
        \partial_t+\partial_s=\partial_{\bar t},\quad \partial_x+\partial_y=\partial_{\bar x},\quad \partial_x-\partial_y=\partial_{\bar y}.
        $$
        We define
        $$
        A_\pm(m):=\frac{A(m)\pm \tilde A(m)}{2},
        $$
        and using the new notation, we rewrite the part into the brackets of the second term in the r.h.s. of \eqref{eq:doubling-Kruzkov} as
        $$
        (A_+(M)-A_+(\tilde M))\partial_{\bar x}w+(A_-(M)-A_-(\tilde M))\partial_{\bar y}w.
        $$
        The part into the brackets in the last term of \eqref{eq:doubling-Kruzkov} can be rewritten as
        $$
        \frac{\phi\ast M+\phi\ast\tilde M}{2}\partial_{\bar x}|M-\tilde M|+\frac{\phi\ast M-\phi\ast\tilde M}{2}\partial_{\bar y}|M-\tilde M|.
        $$
        The result of this substitutions is
        \begin{equation}
             \begin{aligned}
            0 \leq & \iiiint \left[ |M-\tilde M|\partial_{\bar t} w+\sgn(M-\tilde M)(A_+(M)-A_+(\tilde M))\partial_{\bar x}w\right]\,d\bar x\,d\bar t\,d\bar y\,d\bar s\\
            &\hspace{4em}+  \iiiint \sgn(M-\tilde M)(A_-(M)-A_-(\tilde M))\partial_{\bar y}\,d\bar x\,d\bar t\,d\bar y\,d\bar s\\
            &\hspace{4em}+  \frac{1}{2}\iiiint (\phi\ast M + \phi\ast \tilde M)w\partial_{\bar x}|M-\tilde M|\,d\bar x\,d\bar t\,d\bar y\,d\bar s\\
            &\hspace{4em}+  \frac{1}{2}\iiiint (\phi\ast M - \phi\ast \tilde M)w\partial_{\bar y}|M-\tilde M|\,d\bar x\,d\bar t\,d\bar y\,d\bar s.
        \end{aligned}
         \end{equation}
        We have to pass to the limit $\eps\to0$, for this reason we need to treat the $\bar y$ derivatives. Thanks to \cite[Lemma 5.4]{Leslie-Tan} we have that
        $$
        \gamma(M,\tilde M):=\sgn(M-\tilde M)(A_-(M)-A_-(\tilde M))
        $$
        is Lipschitz in both variables and both the Lipschitz constants are less or equal than $|A_-|_{\Lip}$.\newline
        Using \cite[Lemma 5.2 \& Lemma 5.4]{Leslie-Tan}, we obtain
        \begin{equation}\label{eq:gamma}
            \left|\frac{d}{d\bar y}\gamma(M,\tilde M)\right|\leq|A_-|_{\Lip}\left|\partial_1 M(\bar x+\bar y,\bar t+\bar s)\right|
            +|A_-|_{\Lip}\left|\partial_1 M(\bar x-\bar y,\bar t-\bar s)\right|,
        \end{equation}
        where $\partial_1$ denotes differentiation with respect to the spatial variable.
        Furthermore, let us notice that the following estimate also holds true
        \begin{equation}\label{eq:estimate_dy}
             \left| \frac{d}{d\bar y}\left|M(\bar x+\bar y,\bar t+\bar s)-\tilde M(\bar x-\bar y,\bar t-\bar s)\right|\right| 
             \leq \left| \partial_1 M(\bar x+\bar y,\bar t+\bar s)+\partial_1\tilde M(\bar x-\bar y,\bar t-\bar s)\right|.
        \end{equation}
        Using \eqref{eq:gamma}-\eqref{eq:estimate_dy} and letting $\eps\to0$, we end up with
        \begin{align*}
             & \iint \left[ |M( x,t )-\tilde M( x,t)|g(x)h'_\delta(t)+\sgn(M-\tilde M)(A_+(M)-A_+(\tilde M))g'(x)h_\delta(t) \right]\,dxdt\\
            & + \iint |A_-|_{\Lip}(|\partial_x M|+|\partial_x\tilde M|)g(x)h_\delta(t)\,dx\,dt\\
            & + \frac{1}{2}\iint \left[ \phi\ast(M+\tilde M)\partial_x|M-\tilde M| + |\phi\ast (M-\tilde M)\|\partial_x M+\partial_x\tilde M|\right]g(x)h_\delta(t)\,dx\,dt \geq 0
        \end{align*}
        where $(x,t)$ is in place of $(\bar x,\bar t)$ to simplify the notation. Thanks to the previous choice of the functions $g(x)$ and $h_\delta(t)$,
        and taking $\delta\to0$, we get
        \begin{equation}
            \begin{aligned}
             \int_\R|M(x,t)&-\tilde M(x,t)|\,dx \leq \int_\R|M(x,s)-\tilde M(x,s)|\,dx\\
            & + \frac{1}{2}|A-\tilde A|_{\Lip}\int_s^t\int_\R |\partial_x M|+|\partial_x \tilde M|\,dx\,d\tau\\
            & + \frac{1}{2}\int_s^t\int_\R \left[ \phi\ast( M+\tilde M) \cdot \partial_x|M-\tilde M|\right.\\
            &\left.+ |\phi\ast(M-\tilde M)|\cdot \partial_x(M+\tilde M)\right]\,dx\,d\tau.
        \end{aligned}\label{eq:final-stability}
        \end{equation}
        We know that $M$ and $\tilde M$ are nondecreasing, thus we can substitute the terms $|\partial_x M|$, $|\partial_x \tilde M|$ 
        and $|\partial_x M+\partial_x \tilde M|$ with $\partial_x M$, $\partial_x \tilde M$ and $\partial_x M+\partial_x\tilde M$, respectively. We write
        $$
        \frac{1}{2}|A-\tilde A|_{\Lip}\int_s^t\int_\R |\partial_x M|+|\partial_x \tilde M|\,dx\,d\tau=(t-s)|A-\tilde A|_{Lip},
        $$
        while we split the last integral of \eqref{eq:final-stability} into two terms to treat it, 
        the first one is
        \begin{align*}
            \frac{1}{2} \int_s^t\int_\R \phi\ast(M+\tilde M)\partial_x|M-\tilde M|\,dx\,d\tau &\leq \frac{1}{2} \left| \int_s^t\int_\R \phi\ast(M+\tilde M)\partial_x|M-\tilde M|\,dx\,d\tau \right|\\
            & = \frac{1}{2} \left| -\int_s^t\int_\R \phi\ast(\partial_x M+\partial_x \tilde M)|M-\tilde M|\,dx\,d\tau \right| \\
            &\leq \|\phi\|_{L^\infty(\R)}\int_s^t\|M(\cdot,\tau)-\tilde M(\cdot,\tau)\|_{L^1(\R)}\,d\tau,
        \end{align*}
        where we use 
        \begin{align*}
         \left| \phi\ast(\partial_x M+\partial_x \tilde M) \right| & = \left| \int_\R \phi(x-y)(\partial_y M(y,\tau)+\partial_y \tilde M(y,\tau))\,dy \right|\\
        & \leq \|\phi\|_{L^\infty(\R)}\int_\R \partial_y M(y,\tau)+\partial_y \tilde M(y,\tau)\,dy=2\|\phi\|_{L^\infty(\R)}.
        \end{align*}
        Using a similar argument, it is possible to prove that
        $$
        \frac{1}{2}\int_s^t\int_\R |\phi\ast(M-\tilde M)|(\partial_x M+\partial_x \tilde M)\,dx\,d\tau 
        \leq \|\phi\|_{L^\infty}\int_s^t\|M(\cdot,\tau)-\tilde M(\cdot,\tau)\|_{L^1(\R)}\,d\tau.
        $$
        If $s=0$, the inequality becomes
        \begin{equation}
             \begin{aligned}\label{eq:inequality1}
            & \|M(\cdot,t)-\tilde M(\cdot,t)\|_{L^1(\R)}\\
            & \leq \|M^0-\tilde M^0\|_{L^1(\R)}+t|A-\tilde A|_{\Lip}+2\|\phi\|_{L^\infty(\R)}\int_0^t\|M(\cdot,\tau)-\tilde M(\cdot,\tau)\|_{L^1(\R)}\,d\tau.
        \end{aligned}
         \end{equation}
        Applying the Gronwall lemma, we get \eqref{eq:stability-estimate}. The uniqueness follows taking $\tilde M^0=M^0$ and $\tilde A_\eps= A_\eps$.\newline
        Now we prove the estimate \eqref{eq:stability-estimate-two}. As before,
$$
\frac{1}{2}|A-\tilde A|_{\Lip}\int_s^t\int_\R |\partial_x M|+|\partial_x \tilde M|\,dx\,d\tau = |A-\tilde A|_{\Lip}(t-s).
$$
We split the second integral of \eqref{eq:final-stability} into two terms, the first one is
\begin{align*}
     \frac{1}{2} & \int_s^t\int_\R \phi\ast(M+\tilde M)\partial_x|M-\tilde M|\,dx\,d\tau
     =\frac{1}{2}\int_s^t\int_\R \omega\ast(\rho+\tilde\rho)\partial_x|M-\tilde M|\,dx\,d\tau \\ 
     & \leq \abs{\frac{1}{2}\int_s^t\int_\R \omega\ast(\rho+\tilde\rho)\partial_x|M-\tilde M|\,dx\,d\tau}
    \leq \frac{1}{2}\int_s^t\int_\R |\omega\ast(\rho+\tilde\rho)|\abs{|\partial_x|M-\tilde M|}\,dx\,d\tau \\
    & \leq \frac{1}{2} \int_s^t\int_\R 2\|\omega\|_{L^\infty(\R)}\abs{\partial_x|M-\tilde M|}\,dx\,d\tau 
    =\int_s^t\int_\R \|\omega\|_{L^\infty(\R)}|\partial_x(M-\tilde M)|\,dx\,d\tau\\
    & \leq \|\omega\|_{L^\infty(\R)}\int_s^t\int_\R |\partial_x M|+|\partial_x \tilde M|\,dx\,d\tau=2\|\omega\|_{L^\infty(\R)}(t-s).
\end{align*}
While, for the second one, we obtain
\begin{align*}
    \frac{1}{2}  \int_s^t\int_\R |\phi\ast(M-\tilde M)|\partial_x(M+\tilde M)\,dx\,d\tau &\leq \|\omega\|_{L^\infty(\R)} \int_s^t\int_\R \partial_x (M+\tilde M)\,dx\,d\tau\\
    & = 2\|\omega\|_{L^\infty(\R)}(t-s).
\end{align*}
Taking $s=0$, we get
\begin{equation}\label{eq:inequality2}
    \|M(\cdot,t)-\tilde M(\cdot,t)\|_{L^1(\R)} \leq \| M^0-\tilde M^0\|_{L^1(\R)}+(|A-\tilde A|_{\Lip}+4\|\omega\|_{L^\infty(\R)})t,
\end{equation}
and so \eqref{eq:stability-estimate-two} holds true. 
This completes the proof of Theorem~\ref{theorem:entropy-solution}.
\end{proof}

\begin{rem} Combining the two bounds \eqref{eq:inequality1} and \eqref{eq:inequality2}, we have
    \begin{align*}
        &\|M(\cdot,t)-\tilde M(\cdot,t)\|_{L^1(\R)}\\
        &\leq \min\Big\{ e^{2t\|\phi\|_{L^\infty(\R)}}\|M^0-\tilde M^0\|_{L^1(\R)}+\frac{|A-\tilde A|_{\Lip}}{2\|\phi\|_{L^\infty(\R)}}(e^{t\|\phi\|_{L^\infty(\R)}}-1),\\ 
        & \hspace{7em}\qquad 
        \| M^0-\tilde M^0\|_{L^1(\R)} +\left(|A-\tilde A|_{\Lip}+4\|\omega\|_{L^\infty(\R)}\right)t\Big\}\,.
    \end{align*}
    This shows that for large times, the increase of the $L^1$ distance between the solutions is up to linear, while for small times, 
    the first bound provides a better estimate.
\end{rem}

\section{Approximation of the nonlocal system}\label{Sec:4}
\setcounter{equation}{0}

{\blu In this section, we return to the system \eqref{ARZ_conservative} and analyze it using the framework developed in Section~\ref{Sec:3} for the cumulative function $M$.
Specifically, let $\PP_c(\R)$ denote the space of probability measures with compact support, equipped with the Wasserstein-1 metric, which we define below.}

\begin{definition}
    Let $\rho,\mu\in\PP_c(\R)$. The Wasserstein-1 distance between them is
    $$
    \WW_1(\rho,\mu)=\sup_{\Lip(f)\leq1}\left| \int_\R f(x)\,d\rho(x)-\int_\R f(x)\,d\mu(x) \right|.
    $$
\end{definition}
In our setting, the $\WW_1$ convergence is equivalent to the weak$-*$ convergence in the sense of measures.
We remark that, if $M$ and $\tilde M$ are the cumulative distribution functions of $\rho$ and $\mu$ respectively, then {\blu the following identity holds:}
\begin{equation}\label{rem:Wass_distance}
\WW_1(\rho,\mu)=\|M-\tilde M\|_{L^1(\R)}\,.
\end{equation}

Now, let's consider
\begin{equation}\label{eq:system-rho-u}
    \begin{cases}
        \partial_t\rho+\partial_x(\rho u)=0,\\
        \partial_t(\rho u)+\partial_x(\rho u^2)=\rho(\phi\ast(\rho u))-\rho u(\phi\ast\rho),
    \end{cases}
\end{equation}
endowed with initial data
\begin{equation}\label{eq:init-data-rho-u}
{\color{black}    \rho(\cdot,0)=\rho^0,\qquad u(\cdot,0)=u^0.}
\end{equation}

{\color{black} Here we provide the definition of weak solution of the Cauchy problem \eqref{eq:system-rho-u}--\eqref{eq:init-data-rho-u}.}
\begin{definition}\label{definition:weak-solution-rho-u}
    Let $\rho^0\in\PP_c(\R)$ and $u^0\in L^\infty(d\rho^0)$. Define $P^0:=\rho^0 u^0$, which is a signed measure. 
    We say that $(\rho,P)=(\rho,\rho u)$ is a weak solution of the ARZ system \eqref{eq:system-rho-u} with initial data \eqref{eq:init-data-rho-u} if for any $T>0$,
    \begin{itemize}
        \item $\rho\in C([0,T];\PP_c(\R))$.
        \item $P(\cdot,t)$ signed measure for any $t\in[0,T]$. Moreover, $P(\cdot,t)$ is absolutely continuous 
        with respect to $\rho(\cdot,t)$ with the Radon-Nikodym derivative $u(\cdot,t)\in L^\infty(d\rho(t))$, 
        where $u(\cdot,t)d\rho(\cdot,t)=dP(\cdot,t)$, for any $t\in[0,T]$.
        \item $(\rho,u)$ satisfies \eqref{eq:system-rho-u} in the sense of distributions.
        \item The initial condition $(\rho^0,P^0)$ is attained in the following sense: 
        $$
        {\color{black} \rho(0)=\rho^0}\,,\qquad \lim_{t\to0+}\int_\R f(x)\,dP(x,t)=\int_\R f(x)\,dP^0(x) \quad \forall \,   f\in C_c^\infty(\R)\,.     
        $$
    \end{itemize}
\end{definition}
Let us define the pseudo-inverse of a nondecreasing function $M$, defined as
$$
M^{-1}(m):=\inf\{ x\in\R:M(x)\geq m\},\quad m\in (0,1],
$$
which is a left-continuous function. 
In the following theorem we show the entropic selection principle and construct the solution to \eqref{eq:system-rho-u}.
\begin{theorem}\label{th:weak_sol_ARZ}
Assume \textbf{(H)} and \eqref{phi-def}. Let $\rho^0\in\PP_c(\R)$, $u^0\in L^\infty(d\rho^0)$ and $P^0=\rho^0 u^0$. 
We construct a unique \textcolor{black}{pair} $(\rho,P)$ using the following procedure.
    \begin{itemize}
        \item[\textbf{(i)}] Let $M^0(x)=\rho^0((-\infty,x])$ and $\psi^0=u^0+\omega\ast\rho^0$. Define a Lipschitz flux $$A:[0,1]\to\R$$ such that
        \begin{equation}\label{def:A-eps-integral-formula}
            A(m)=\int_0^m a(m')\,dm',\qquad m\in[0,1],
        \end{equation}
        with $a(m)=\psi^0\circ(M^0)^{-1}(m)$.
        \item[\textbf{(ii)}] Let $M$ be the unique entropy weak solution of \eqref{def:Cauchy_problem} with initial datum $M^0$ and flux $A$.
        \item[\textbf{(iii)}] Define $(\rho,P)$ using $M$ in the following way
        \begin{equation}\label{def:rho-P-via-M}
            \rho=\partial_x M,\qquad P=-\partial_t M=\partial_x(A\circ M)-(\phi\ast M)\partial_x M.
        \end{equation}
    \end{itemize}
Then, the \textcolor{black}{pair} $(\rho,P)$ is a weak solution of the ARZ system \eqref{eq:system-rho-u} with initial data \eqref{eq:init-data-rho-u} 
in the sense of Definition \ref{definition:weak-solution-rho-u}. Moreover, we can define $u(\cdot,t)=\frac{dP(\cdot,t)}{d\rho(\cdot,t)}$ 
that is the Radon-Nikodym derivative of $P(t)$ with respect to $\rho(t)$.
\end{theorem}
\begin{proof}
The functions $M^0$ and $A$ defined in the step \textbf{(i)} satisfy the hypotheses of Theorem \ref{theorem:entropy-solution}.
Indeed, $M^0$ is a nondecreasing function  by construction, and $\rho^0$ is a probability measure with compact support, 
then it is possible to find an $R^0>0$ such that $\supp \{\rho^0\}\subset[-R^0,R^0]$. Regarding $A$, it is possible to see that
    $$
    \|\psi^0\|_{L^\infty(d\rho^0)} \leq \|u^0\|_{L^\infty(d\rho^0)}+\|\omega\|_{L^\infty(\R)}<+\infty,
    $$
as a consequence $A$ as in \eqref{def:A-eps-integral-formula} is Lipschitz continuous. 
Applying Theorem \ref{theorem:entropy-solution} we obtain a unique entropy weak solution 
of the scalar balance problem \eqref{def:Cauchy_problem}, $M\in BV(\R\times[0,T])$ for a fixed $T>0$. 

To verify that the \textcolor{black}{pair} $(\rho,P)$ constructed as in \eqref{def:rho-P-via-M} satisfies the Definition \ref{definition:weak-solution-rho-u} 
of weak solution, we apply \cite[Lemma 2.1]{Brenier-Grenier}, then the convergence result \eqref{eq:M_N-to-M} gives $\rho\in C([0,T];\PP_c(\R))$, 
while \eqref{eq: dev-M_N-to-M} implies that $P=-\partial_t M\in \MM(\R)$. Since $M\in BV(\R\times[0,T])$ and $A$ is a Lipschitz function, 
using Theorem \ref{th:chain-rule}, we know that there exists a bounded Borel measurable function $\psi=\psi(x,t)$ such that
    $$
    \partial_x(A(M))=\psi\partial_x M,\qquad \partial_t(A(M))=\psi\partial_t M,
    $$
where $\psi=A'$ almost everywhere for the Lebesgue measure on $\R$, moreover $\psi$ is bounded by $|A|_{\Lip}$. Then,
    $$
    P=-\partial_t M=\partial_x(A(M))-(\phi\ast M)\partial_x M=(\psi-(\omega\ast \rho))\rho=u\rho,
    $$
    where $u=\psi-\omega\ast\rho$. Now, we want to verify that $(\rho,u)$ satisfies system \eqref{eq:system-rho-u} in the sense of distributions
    \begin{align*}
        &\partial_t\rho=\partial_t(\partial_x M)=\partial_x(\partial_t M)=\partial_x(-P)=-\partial_x(\rho u),\\
        & \partial_t(\rho\psi)=\partial^2_{tx}(A(M))=\partial_x(\psi\partial_t M)=-\partial_x(\psi P)=-\partial_x(\rho u \psi).
    \end{align*}
    The first equation of \eqref{eq:system-rho-u} is verified, while for the second one we have
    \begin{align*}
        \partial_t(\rho u) & +\partial_x(\rho u^2)= \partial_t(\rho(\psi-\omega\ast\rho))+\partial_x(\rho u(\psi-\omega\ast\rho))\\
        & = \underbrace{\partial_t(\rho \psi)+\partial_x(\rho u\psi)}_{=0} -\partial_t(\rho(\omega\ast\rho))-\partial_x(\rho u(\omega\ast\rho))\\
        & = -\underbrace{(\partial_t\rho+\partial_x(\rho u))(\omega\ast\rho)}_{=0}-\rho\partial_t(\omega\ast\rho)-\rho u \partial_x(\omega\ast\rho)\\
        & = -\rho(\omega\ast\partial_t\rho)-\rho u(\omega\ast\partial_x\rho)\\
        & = \rho(\omega\ast\partial_x(\rho u))-\rho u (\omega\ast\partial_x\rho)=\rho(\phi\ast(\rho u))-\rho u (\phi\ast\rho).
    \end{align*}
    Finally, we have to verify if the initial data are attained in the sense of Definition \ref{definition:weak-solution-rho-u}. 
    By the continuity of $\rho$ in time, we obtain the first condition about $\rho^0$, while for $P$, using \eqref{def:rho-P-via-M}, we obtain
    $$
    \int_\R f(x)\,dP(x,t)=\int_\R f(x)\,dA(M(x,t))-\int_\R f(x)(\phi\ast M)(x,t)\,d\rho(x,t).
    $$
    By the fact that $A$ is Lipschitz, the first integral of the r.h.s. is equal to
    $$
    \int_\R f(x) \partial_x A(M(x,t))\,dx=-\int_\R f'(x)A(M(x,t))\,dx.
    $$
    Now, we pass to the limit as $t\to0$, for the term above, we have
    $$
    \abs{\int_\R f'(x)A(M(x,t))\,dx-\int_\R f'(x)A(M^0(x))\,dx}\leq |A|_{Lip}\|f'\|_{L^\infty}\|M(\cdot,t)-M^0\|_{L^1}\to0.
    $$
    For the second term we can write
    \begin{align*}
        &\abs{\int_\R f(x)(\phi\ast M)(x,t)\,d\rho(x,t)-\int_\R f(x)(\phi\ast M^0)(x)\,d\rho^0(x)}\\
        &\leq \|\phi\ast(M-M^0)\|_{L^\infty(\R)}\abs{\int_\R f(x)\,d\rho(x,t)-\int_\R f(x)\,d\rho^0(x)}\\
        &\leq \|\phi\|_{L^\infty(\R)}\|M-M^0\|_{L^1(\R)}\abs{\int_\R f(x)\,d\rho(x,t)-\int_\R f(x)\,d\rho^0(x)}\to 0,
    \end{align*}
    thanks to the convergences of $\rho$ and $M$.
    This concludes the proof.
\end{proof}
Thanks to Theorem \ref{th:weak_sol_ARZ}, we can conclude that the solution of the ARZ system \eqref{eq:system-rho-u} 
exists and it is unique; moreover it satisfies the stability estimate established in the following proposition.
\begin{prop}
    Let $\rho^0,\mu^0\in\PP_c(\R)$, $u^0\in L^\infty(d\rho^0)$, $v^0\in L^\infty(d\mu^0)$. Consider the weak solutions $(\rho,u)$ and $(\mu,v)$ 
    constructed as in Theorem \ref{th:weak_sol_ARZ} with initial data $(\rho^0,u^0)$ and $(\mu^0,v^0)$ respectively. 
    Then, $\rho$ and $\mu$ satisfy the following stability estimate
    \begin{equation}\label{eq:stab_estimate_rho}
        \WW_1(\rho,\mu)\leq \WW_1(\rho^0,\mu^0)e^{2t\|\phi\|_{L^\infty(\R)}}+\frac{|A-\tilde A|_{Lip}}{2\|\phi\|_{L^\infty(\R)}}(e^{t\|\phi\|_{L^\infty(\R)}}-1).
    \end{equation}
\end{prop}
\begin{proof} {\color{black} By means of  \eqref{rem:Wass_distance}, 
the stability estimate \eqref{eq:stability-estimate} rewrites as} 
    \begin{align*}
    \WW_1(\rho,\mu)&=\|M-\tilde M\|_{L^1(\R)}\\
    &\leq e^{2t\|\phi\|_{L^\infty(\R)}}\|M^0-\tilde M^0\|_{L^1(\R)}+\frac{|A-\tilde A|_{Lip}}{2\|\phi\|_{L^\infty(\R)}}(e^{t\|\phi\|_{L^\infty(\R)}}-1)\\
   & =\WW_1(\rho^0,\mu^0)e^{2t\|\phi\|_{L^\infty(\R)}} + \frac{|A-\tilde A|_{Lip}}{2\|\phi\|_{L^\infty(\R)}}(e^{t\|\phi\|_{L^\infty(\R)}}-1)
    \end{align*}
     {\color{black} which is \eqref{eq:stab_estimate_rho}.}
\end{proof}
In what follows we show that it is possible to obtain a solution of system \eqref{eq:system-rho-u} using atomic initial data constructed 
using the sticky particle Cucker-Smale dynamics \eqref{eq:C-S}-\eqref{eq:C-S omega}-\eqref{def:common-velocity}, {\color{black} the following results hold}. 
\begin{prop}
    Let $N\in\N$ and consider system \eqref{eq:system-rho-u} with initial data
    $$
    \rho_N^0(x)=\sum_{j=1}^N m_{j,N}\delta(x-x_{j,N}^0),\qquad P_N^0(x)=\sum_{j=1}^N m_{j,N}v_{j,N}^0\delta(x-x_{j,N}^0).
    $$
Where $x_{j,N}^0\in [-R^0,R^0],$ for a constant $R^0>0,$ and $\sum_{j=1}^N m_{j,N}=1$. 
Let $(x_{j,N}(t),v_{j,N}(t))$ be the solution of the sticky particle Cucker-Smale dynamics \eqref{eq:C-S}-\eqref{eq:C-S omega}-\eqref{def:common-velocity} 
with initial data $(x_{j,N}^0,v_{j,N}^0,m_{j,N})_{j=1}^N$, then the solution of system \eqref{eq:system-rho-u} is
    $$
    \rho_N(x,t)=\sum_{j=1}^N m_{j,N}\delta(x-x_{j,N}(t)),\qquad P_N(x,t)=\sum_{j=1}^N m_{j,N}v_{j,N}(t)\delta(x-x_{j,N}(t)).
    $$
\end{prop}
\begin{theorem}
    Let $(\rho^0,P^0)$ satisfy the hypotheses of Theorem \ref{th:weak_sol_ARZ} and let $(\rho,P)$ let be the associated weak solution. 
    There exists a sequence of atomic initial data $(\rho_N^0,P_N^0)_{N=1}^\infty$ such that the sequence of the associated solutions satisfies
    \begin{align*}
    \WW_1(\rho_N(t) & ,\rho(t))\to 0,\\
    \qquad P_N(t) \weakstar P(t) & \qquad\text{in $\MM(\R)$},
    \end{align*}
    for any $t>0$, for $N\to\infty$.
\end{theorem}
{\color{black} The above results correspond to Proposition 6.4 and Theorem 6.5 in \cite{Leslie-Tan}. Since their proofs do no require additional consideration, we omit them.}

\appendix
\section{Convolution of measures and functions}\label{sec:convolution}
\setcounter{equation}{0}
In this section we recall the theory about the convolution product between a function and a measure, 
in particular we want to see when this kind of operation is well defined. 
We refer to \cite{AmbrosioFuscoPallara,Bogachev} 
for this section. We start with some definitions.
\begin{definition}
    Let $X$ be a locally compact and separable metric space, $\mathcal{B}(X)$ its Borel $\sigma$-algebra, 
    and consider the measure space $(X,\mathcal{B}(X))$. A positive measure on $(X,\mathcal{B}(X))$ is called 
    a Borel measure. If a Borel measure is finite on the compact sets, it is called a positive Radon measure.
\end{definition}
\begin{definition}
    Let $X,Y$ be metric spaces, and let $f:X\to Y$. We say that $f$ is a Borel function if 
    $f^{-1}(A)\in\mathcal{B}(X)$ for any $A\subset Y$ open set.
\end{definition}
Now we can give the proposition which defines the convolution product between measures and functions.
\begin{prop}
    Let $f$ be a Borel function in $L^p(\R^n)$ and let $\mu$ be a bounded Borel measure on $\R^n$. The function
    $$f\ast\mu(x):=\int_{\R^n} f(x-y)\,d\mu(y),
    $$
    is defined for almost all $x$ with respect to Lebesgue measure and
    $$
    \|f\ast\mu\|_{L^p(\R^n)}\leq \|f\|_{L^p(\R^n)}\|\mu\|.
    $$
\end{prop}
We want to see if it is possible to apply this theory to our case, in particular we want to verify 
if the convolution product between the function $\omega$ and the probability measure $\rho$ is well defined.
We start with the case of the measure $\rho_N=\sum_{i=1}^N m_i\delta(x-x_i)$ with $\sum_{i=1}^N m_i=1$. 
The $\rho_N$ is a sum of Dirac delta, thus we can verify if the product is well defined using only a Dirac.
We need to verify if the hypotheses of the function and of the measure are satisfied. The delta is a Radon measure, 
in particular is a probability measure, and so it is a bounded Borel measure, thus we have to verify only if 
$\omega$ is a Borel function with respect to the $\sigma$-algebra generated by the Euclidean topology and 
this is true, as $\omega$ is continuous.

The following convolution product is defined for almost all $x$ with respect to Lebesgue measure and
$$
\omega\ast\delta(x)=\int_\R\omega(x-y)\,d\delta(y).
$$
The next step is to verify if the following equality holds true
\begin{equation}\label{eq:omega_conv_rho_N}
\omega\ast\rho_N(x,t)=\phi\ast M_N(x,t).
\end{equation}
We consider a single Dirac delta as before, if the equality before holds true for a single delta 
we can extend the result to the sum of delta's. Thus we want to see if, with $\rho_N=\delta$ and $M_N=H$, 
where $H$ is the right continuous Heaviside function, the following equality is verified
\begin{equation}\label{eq:omega_convolution}
\omega(x)=\int_\R\omega(x-y)\,d\rho(y)=\omega\ast\delta(x)=\phi\ast H(x).
\end{equation}
We want to evaluate the following integral
\begin{align*}
\phi\ast H(x) & =\int_\R\phi(x-y)H(y)\,dy = \int_0^\infty \phi(x-y)\,dy\\
& = -\omega(x-y)|_0^\infty=\omega(x).
\end{align*}
We can conclude that (\ref{eq:omega_convolution}) holds true and then \eqref{eq:omega_conv_rho_N} holds true too.

\section{Chain Rule for functions of Bounded Variation}
The results in this section are taken from \cite{DalMaso-LF-Murat}.
Let $\Omega$ be a bounded open subset of $\R^n$ and let $u\in BV(\Omega,\R)$. 
Let us recall that $x\in\Omega$ is a Lebesgue point of $u$ if there exists a real number $\bar u(x)$ such that
$$
\lim_{\rho\to0}\frac{1}{|B_\rho(x)|}\int_{B_\rho(x)}|u(y)-\bar u(x)|\,dy=0,
$$
with $B_\rho(x)=\{y\in \R^n:|y-x|<\rho\}$ and $|B_\rho(x)|$ denotes the Lebesgue measure of $B_\rho(x)$. 
The set of all Lebesgue points is denoted by $\Omega_u$. A point $x\in\Omega\setminus\Omega_u$ is defined 
a \emph{jump point} of $u$ if there exist two real numbers $u_-(x)$ and $u_+(x)$ and a unit vector $v_u(x)\in\R^n$ 
such that $u_+(x)\neq u_-(x)$ and
$$
\lim_{\rho\to0}\frac{1}{|B^\pm_\rho(x)|}\int_{B^\pm_\rho(x)}|u(y)-u_\pm(x)|\,dy=0,
$$
where $B^\pm_\rho(x)=B_\rho(x)\cap\{y\in\R^n:(y-x,\pm v_u(x))>0\}.$ The set of all jump points of $u$ is denoted by $S_u$. 
Let $g:\R\to\R$ be a locally bounded Borel function, the averaged superposition $\hat{g}(u)$ of $g$ and $u$ is defined by
\begin{equation}\label{def:superposition}
    \hat{g}(u)(x)=\begin{cases}
        g(u(x))\qquad x\in\Omega_u,\\
        \int_0^1 g(u_-(x)+s(u_+(x)-u_-(x)))\,ds\qquad x\in S_u.
    \end{cases}
\end{equation}
\begin{theorem}\label{th:chain-rule}
Let $\Omega$ be a bounded open subset of $\R^n$, and $f:\R\to\R$ be a globally Lipschitz continuous function. 
Consider $u\in BV(\Omega,\R)$ and set $v=f\circ u$. Then the function $v$ belongs to $BV(\Omega,\R)$ and we have the chain rule
$$
Dv=\hat{f^*}(u)Du,\quad\text{as a Borel measures on $\Omega$},
$$
where $f^*:\R\to\R$ denotes an arbitrary bounded Borel function such that
$$
f^*(t)=f'(t),\quad\text{almost everywhere for the Lebesgue measure on $\R$}.
$$
\end{theorem}

\section*{Declarations}

- {\bf Availability of data and material}: Not applicable.

- {\bf Competing interests}: The authors declare that they have no competing interests.

- {\bf Funding}: The authors were partially supported by the GNAMPA 2025 project \emph{"Analisi e controllo di modelli evolutivi con fenomeni non locali"} and by the project \emph{"Leggi di conservazione con termini nonlocali e applicazioni al traffico veicolare"} 
(Progetto per la ricerca di base 2025, University of L'Aquila).

- {\bf Authors' contributions}: All authors have contributed equally to the work. All authors read and approved the final manuscript.

- {\bf Acknowledgements}: 
{The authors wish to thank Prof. A.~\'{S}wierczewska-Gwiazda for pointing out references \cite{CPSZ2024,WZ2025}. They also wish to thank the two referees for their careful reviews and insightful comments.}

\end{document}